\documentclass[11pt]{amsart}

\usepackage{amsmath}
\usepackage{amssymb,mathtools,amsthm,latexsym}
\usepackage{mathrsfs}
\usepackage{bm}
\usepackage{url}
\usepackage{xcolor}
\numberwithin{equation}{section}

\usepackage[nocompress]{cite}

\calclayout

\usepackage{comment}
\usepackage{soul}

\newtheorem{thm}{Theorem}[section]
\newtheorem{lemma}[thm]{Lemma}
\newtheorem{cor}[thm]{Corollary}
\theoremstyle{definition}
\newtheorem{defn}[thm]{Definition}

\newtheorem{remark}[thm]{Remark}
\newtheorem{prop}[thm]{Proposition}

\title[Curvewise characterizations and a Sobolev Differential]{Curvewise characterizations of minimal upper gradients and the construction of a Sobolev differential}
\author{Sylvester Eriksson-Bique}
\author{Elefterios Soultanis}

\address{Sylvester Eriksson-Bique\\
Research Unit of Mathematical Sciences \\
P.O.Box 3000\\
FI-90014 Oulu}
\email{\tt sylvester.eriksson-bique@oulu.fi}

\address{Elefterios Soultanis\\
Department of Mathematics\\
University of Fribourg, 1700 Fribourg, Chemin du Musee 23, Switzerland}
\email{\tt elefterios.soultanis@gmail.com}

\def\Xint#1{\mathchoice
	{\XXint\displaystyle\textstyle{#1}}%
	{\XXint\textstyle\scriptstyle{#1}}%
	{\XXint\scriptstyle\scriptscriptstyle{#1}}%
	{\XXint\scriptscriptstyle\scriptscriptstyle{#1}}%
	\!\int}
\def\XXint#1#2#3{{\setbox0=\hbox{$#1{#2#3}{\int}$}
		\vcenter{\hbox{$#2#3$}}\kern-.5\wd0}}

\def\dashint{\Xint-}

\newcommand{\N}{\ensuremath{\mathbb{N}}}
\newcommand{\R}{\ensuremath{\mathbb{R}}}
\newcommand{\Q}{\ensuremath{\mathbb{Q}}}

\newcommand{\LIP}{\ensuremath{\mathrm{LIP}}}

\newcommand{\diam}{\ensuremath{\ \mathrm{diam}}}
\newcommand{\Lip}{\ensuremath{\mathrm{Lip}}}
\newcommand{\lip}{\ensuremath{\mathrm{lip}}}
\newcommand{\defeq}{\mathrel{\mathop:}=}
\newcommand{\Diff}{\mathrm{Diff}}

\newcommand{\len}{\ensuremath{\mathrm{Len}}}

\newcommand{\Mod}{\ensuremath{\mathrm{Mod}}}

\newcommand{\eps}{\varepsilon}

\newcommand{\ud}{\mathrm{d}}
\newcommand{\Ne}[2]{N^{1,#1}(#2)}
\newcommand{\inv}{^{-1}}

\newcommand{\esssup}{{\rm ess \,\, sup}}
\newcommand{\essinf}{{\rm ess \,\, inf}}

\newcommand{\spt}{\mathrm{spt}}

\begin{document}

\begin{abstract}
	We represent minimal upper gradients of  Newtonian functions, in the range $1\le p<\infty$, by maximal directional derivatives along ``generic'' curves passing through a given point, using plan-modulus duality and disintegration techniques. As an application we introduce the notion of $p$-weak charts and prove that every Newtonian function admits a differential with respect to such charts, yielding a linear approximation along $p$-almost every curve. The differential can be computed curvewise, is linear, and satisfies the usual Leibniz and chain rules.
	
	The arising $p$-weak differentiable structure exists for spaces with finite Hausdorff dimension and agrees with Cheeger's structure in the presence of a Poincar\'e inequality. It is moreover compatible with, and gives a geometric interpretation of, Gigli's abstract differentiable structure, whenever it exists. The $p$-weak charts give rise to a finite dimensional $p$-weak cotangent bundle and pointwise norm, which recovers the minimal upper gradient of Newtonian functions and can be computed by a maximization process over generic curves. As a result we obtain new proofs of reflexivity and density of Lipschitz functions in Newtonian spaces, as well as a characterization of infinitesimal  Hilbertianity in terms of the pointwise norm.

\end{abstract}
\maketitle
\tableofcontents

\section{Introduction}
\subsection{Overview} Minimal weak upper gradients of Sobolev type functions on metric measure spaces were first introduced by Cheeger \cite{che99}, building on the notion of upper gradients from \cite{hei98}. Shanmugalingam \cite{sha00} developed \emph{Newtonian spaces} $\Ne pX$ using the modulus perspective of \cite{hei98}, and proved that they coincide with the Sobolev space defined by Cheeger up to modification of its elements on a set of measure zero. Further notions of Sobolev spaces, based on test plans, were developed by Ambrosio--Gigli--Savare \cite{AGS14}, with a corresponding notion of minimal gradient. Earlier, Haj\l asz \cite{haj96} had introduced a Sobolev space whose associated minimal gradient however lacks suitable locality properties. While the various Sobolev spaces (with the exception of Haj\l asz's definition) are equivalent for generic metric measure spaces, Newtonian spaces consist of representatives which are absolutely continuous along generic curves, a property central to the results in this paper.

The minimal $p$-weak upper gradient $g_f\in L^p(X)$ of a Newtonian function $f\in N^{1,p}(X)$ on a metric measure space $X$ is a Borel function characterized (up to a null-set) as the minimal function satisfying 
\begin{equation}\label{eq:ug}
|(f\circ\gamma)'_t|\le g_f(\gamma_t)|\gamma_t'| \quad\textrm{ for a.e. }t\in I
\end{equation}
for all absolutely continuous $\gamma:I\to X$ outside a curve family of zero $p$-modulus. Here $|\gamma_t'|$ denotes the metric derivative of $\gamma$ for a.e. $t$, see Section \ref{sec:prelims}. When $X=\R^n$ and $f\in C^\infty_c(\R^n)$, $g_f$ is given by $g_f=\|\nabla f\|$; in this case, for each $x\in X$, there exists a (smooth) gradient curve $\gamma:(-\eps,\eps)\to X$ with $\gamma_0=x$, satisfying
\begin{align}\label{eq:revug}
(f\circ\gamma)_0'=g_f(x)|\gamma_0'|.
\end{align}

In general however, despite the minimality of $g_f$, the equality in \eqref{eq:revug} is not always attained. For example the fat 
Sierpi\'nski carpet (with the Hausdorff 2-measure and Euclidean metric) constructed in \cite{mac13} with a sequence in $\ell^2\setminus\ell^1$, as pointed out in the introduction of \cite{mac13}, gives zero $p$-modulus ($p>1$) to the family of curves parallel to the $x$-axis, and thus to the family of gradient curves of the function $f(x,y)=x$. 
We remark that the example above is measure doubling and supports a Poincar\'e inequality; in this context an approximate form of \eqref{eq:revug} for Lipschitz functions was proven in \cite[Theorem 4.2]{cheegerkleiner}.

Towards a positive answer for generic spaces, an ``integral formulation'' of \eqref{eq:revug} given by a result of Gigli \cite[Theorem 3.14]{gig15} states that, when $p>1$ and $f\in \Ne pX$, there exist probability measures $\bm\eta$ on $C(I;X)$ (known as test plans representing the gradient of $f$) such that 
\begin{align*}
\lim_{t\to 0}\int\frac{f(\gamma_t)-f(\gamma_0)}{t}\ud\bm\eta=\lim_{t\to 0}\int\frac 1t\int_0^tg_f(\gamma_s)|\gamma_s'|\ud s\ud\bm\eta.
\end{align*}

In this paper we obtain a ``pointwise'' variant of \eqref{eq:revug} for general metric measure spaces using a combination of plan-modulus duality, developed in \cite{amb13,shanmunsemmes,exnerova2019plans}, and disintegration techniques. Theorem \ref{thm:curvewise-p=1} below expresses the minimal weak upper gradient of a Newtonian function as the supremum of directional derivatives along generic curves passing through a given point. Here, it is crucial to use Newtonian functions, which are absolutely continuous along almost every curve.

This curvewise characterization of minimal upper gradients yields the existence of an abundance of curves in a given region of the space, provided the region supports non-trivial Newtonian functions. The idea of constructing an abundance of curves goes back to Semmes \cite{sem96} in the presence of a Poincar\'e inequality. Under this assumption Cheeger showed that $g_f=\Lip f$, where $\Lip f$ denotes the pointwise Lipschitz constant of a Lipschitz function $f$.  Note that inequality $\Lip f\le g$ for \emph{continuous} upper gradients $g$ of a Lipschitz function $f$ on a geodesic space is a direct, but central, observation made in \cite[p. 432--433]{che99}. 

The work of Cheeger lead to many developments, including the paper \cite{cheegerkleiner} pioneering the idea of using directional derivatives along curves (and the early version of Theorem \ref{thm:curvewise-p=1} in \cite[Theorem 4.2]{cheegerkleiner}) as well as the development and detailed analysis of \emph{Lipschitz differentiability spaces}, cf. \cite{keith04,bate12diff,bate2013differentiability,CKS,sch16,sch16b}. In the latter, curves are replaced by \emph{curve fragments} whose abundance is expressed using \emph{Alberti representations}. Alberti representations are similar to \emph{plans} used in this paper. The connection between such representations and the ideas in \cite{cheegerkleiner} was first observed by Preiss, see \cite[p.2]{bate12diff}, and can be used to prove the self-improvement of the $\Lip-\lip$ inequality to the $\Lip-\lip$ equality, cf. \cite{bate12diff,sch16b,CKS}.

Similarly the abundance of curves, obtained here using duality, yields geometric information on \emph{Sobolev} functions on \emph{general} metric measure spaces. (Indeed, duality, in the disguise of a minimax principle, was previously used to find Alberti representations in Lipschitz differentiability spaces, see \cite[Theorem 5.1]{bate12diff} which uses \cite[Lemma 9.4.3]{rud}.) As an important first application, we use curvewise directional derivatives to define \emph{$p$-weak charts} and a differential for Newtonian functions with respect to such charts. The arising \emph{$p$-weak differentiable structure}, i.e. a covering by $p$-weak charts, exists for any Lipschitz differentiability space but also in more general situations, see  Proposition \ref{prop:exists-p-weak-diff-str}. With the aid of Theorem \ref{thm:curvewise-p=1} we adapt Cheeger's construction to produce a measurable $L^\infty$-bundle, called the \emph{$p$-weak cotangent bundle}, over spaces admitting a $p$-weak differentiable structure; differentials of Newtonian functions are sections over this bundle. While the Cheeger differential $\ud_C f$ yields a linearization of a Lipschitz function $f$, our $p$-weak differential $\ud f$ is given by a linearization along $p$-almost every curve, and the pointwise norm of $\ud f$ recovers the minimal weak upper gradient, cf. Theorem \ref{thm:differentiability}.

This definition of a weak differentiable structure seems to be the natural extension of Cheeger's seminal work in \cite{che99} to settings without a Poincar\'e inequality and yields a "partial differentiable structure", which has been the aim of many authors previously, cf. \cite{luvcic2020characterisation,almar16,sch16b, CKS}. Namely, the $p$-weak cotangent bundle measures and makes precise the set of accessible directions (for positive modulus) in the space. By constructing the differential using directional derivatives along curves, we give it a concrete geometric meaning. A sequence of recent work has sought such concrete descriptions, see e.g. \cite{tet2020,luvcic2020characterisation}.  Our approach yields a new unification of the concrete and abstract cotangent modules of Cheeger \cite{che99} and Gigli \cite{gig18}, respectively; the $p$-weak cotangent bundle is compatible with Gigli's cotangent module when the latter is locally finitely generated, and with Cheeger's cotangent bundle when the space satisfies a Poincar\'e inequality, see Theorems \ref{thm:gigli} and \ref{thm:PI}. 

The geometric approach in this paper has many natural applications. We mention here the \emph{tensorization of Cheeger energy},  pursued in \cite{ APS14, AGS14,luvcic2020characterisation}, and the identification of abstract tangent bundles with geometric tangent cones, cf. \cite{almar16, luvcic2020characterisation}. Our methods give tools to generalize and refine the results mentioned above, and moreover enable a blow-up analysis to study analogues of \emph{generalized linearity} considered for example in \cite{che99, CKS}. Indeed, blow-ups of plans that define the pointwise norm on a $p$-weak chart (see Lemmas \ref{lem:canonicalplan}, \ref{lem:properties} and \ref{lem:representative}) along a sequence of re-scaled spaces yield curves in the limiting space along which limiting maps of rescaled Newtonian maps behave linearly. In this context we highlight \cite{sch16b}, which gives a similar geometric and blow-up analysis in the context of abstract Weaver derivations. We leave the detailed exploration and development of these ideas for future work.

\subsection{Curvewise characterization of minimal upper gradients}

Throughout the paper, we fix a \emph{metric measure space} $X=(X,d,\mu)$, that is, a complete separable metric space $(X,d)$ together with a Radon measure $\mu$ which is finite on bounded sets. A \emph{plan} is a finite measure $\bm\eta$ on $C(I;X)$ that is concentrated on the set $AC(I;X)$ of absolutely continuous curves. The  natural evaluation map $e:C(I;X)\times I\to X,\ (\gamma,t)\mapsto\gamma_t$ gives rise to the \emph{barycenter} $\ud\bm\eta^\#:=e_\ast(|\gamma_t'|\ud t\ud\bm\eta)$ of $\bm\eta$. If $\ud\bm\eta^\#=\rho\ud\mu$ for some $\rho\in L^q(\mu)$, we say that $\bm\eta$ is a $q$-plan. (Not to be confused with $q$-test plans, see e.g. Section \ref{sec:prelims} or \cite{amb13}.) Every finite measure $\bm\pi$ on $C(I;X)\times I$ admits a disintegration with respect to $e$: for $e_\ast\bm\pi$-almost every $x\in X$, there exists a (unique) measure $\bm\pi_x\in \mathcal P(C(I;X)\times I)$, concentrated on $e\inv(x)$, such that the collection $\{\bm\pi_x\}$ satisfies

\[
\bm\pi(B)=\int \bm\pi_x(B)\ud(e_*\bm\pi)(x)
\]
for all Borel sets $B\subset C(I;X)\times I$. We refer to \cite{AGS08,bogachev07} for more details.

We use these notions to define a ``generic curve'': if $\bm\eta$ is a $q$-plan on $X$ and $\{\bm\pi_x\}$ the disintegration of $\ud\bm\pi:=|\gamma_t'|\ud t\ud\bm\eta$ with respect to $e$, then $\bm\pi_x$-a.e.-curve passes through $x$, for $e_*\bm\pi$-a.e. $x\in X$. In the forthcoming discussion, we omit the reference to $e$ in the disintegration. We now formulate our first result, in which the equality in \eqref{eq:revug} is obtained as an essential supremum with respect to the disintegration for almost every point. In the statement below we denote
\[
\Diff(f)=\{ (\gamma,t)\in AC(I;X)\times I: f\circ\gamma\in AC(I;\R), (f\circ\gamma)_t'\textrm{ and }|\gamma_t'|>0 \textrm{ exist} \}
\]
for a $\mu$-measurable function $f:X\to [-\infty,\infty]$.

\begin{thm}\label{thm:curvewise-p=1}
	Let $1\le p<\infty$, and let $1<q\le \infty$ satisfy $1/p+1/q=1$. Suppose $f\in N^{1,p}(X)$, $g_f$ is a Borel representative of the minimal $p$-weak upper gradient of $f$, and $D:=\{ g_f>0 \}$. There exists a $q$-plan $\bm\eta$ with $\mu|_D \ll \bm\eta^\#$ so that the disintegration $\{\bm\pi_x\}$ of $\ud\bm\pi:=|\gamma_t'|\ud t\ud\bm\eta$ is concentrated on $ e\inv(x)\cap \Diff(f)$ and
		\begin{align}\label{eq:almost-attained-p=1}
		g_f(x)= \left\|\frac{(f\circ \gamma)'_t}{|\gamma'_t|}\right\|_{L^\infty(\bm\pi_x)}
		\end{align}
	for $\mu$-almost every $x\in D$.
\end{thm}

\begin{remark}\label{rmk:local} The statement also holds when $f\in N^{1,p}_{loc}(X)$, that is, $f|_{B(x,r)}\in N^{1,p}(B(x,r))$ for each ball $B(x,r) \subset X$. Indeed, a localization argument, replacing $f$ by $f\eta_n$ with $\eta_n$ a sequence of Lipschitz functions with bounded support and $\eta_n|_{B(x_0,n-1)}=1$ for some $x_0$, reduces the statement for $f\in N^{1,p}_{loc}(X)$ to Theorem \ref{thm:curvewise-p=1}. Similarly, other notions in this paper, such as charts, could use a local Sobolev space, but to avoid technicalities we do not discuss this point further.  A reader can see Lemma \ref{lem:difflip} and its proof for a prototypical form of such a localization argument.
\end{remark}

In particular, we have the following corollary.
\begin{cor}\label{cor:integral}
Let $p,q$ and $f, g_f, D$ be as in Theorem \ref{thm:curvewise-p=1}. There exists a $q$-plan $\bm\eta$ and, for every $\varepsilon>0$, a Borel set $B=B_\varepsilon\subset \Diff(f)$ with the following property. If $\{\bm\pi_x\}$ denotes the disintegration of $\ud\bm\pi:=|\gamma_t'|\ud t\ud\bm\eta$,  then $\bm\pi_x(B)>0$ and
\begin{align*}
(1-\varepsilon)g_f(x)|\gamma_t'|\le (f\circ\gamma)_t'\le g_f(x)|\gamma_t'| \quad\textrm{ for every }(\gamma,t)\in e\inv(x)\cap B
\end{align*}
for $\mu$-a.e. $x\in D$.
\end{cor}

Theorem \ref{thm:curvewise-p=1} notably covers the case $p=1$. In Section \ref{subsec:pgtr1plan} we also prove a variant (Theorem \ref{thm:curvewise}) when $p>1$, using test plans representing a gradient instead of plan-modulus duality.


\subsection{Application: $p$-weak differentiable structure}

Cheeger \cite{che99} showed that PI-spaces (metric measure spaces with a doubling measure supporting some Poincar\'e inequality) admit a countable cover by \emph{Cheeger charts}, also called a \emph{Lipschitz differentiable structure} (see \cite{kei04}). Let $\LIP(X)$ denote the collection of Lipschitz functions on $X$, and let $\LIP_b(X)$ consist those Lipschitz functions with bounded support. 
A Cheeger chart $(U,\varphi)$ of dimension $n$ consists of a Borel set $U$ with $\mu(U)>0$, and a Lipschitz function $\varphi:X\to \R^n$ such that, for every $f\in \LIP(X)$ and $\mu$-a.e. $x\in U$ there exists a unique linear map $\ud_{C,x}f:\R^n\to\R$, called the Cheeger differential of $f$, such that
\begin{align}\label{eq:cheegerdiff}
f(y)-f(x)=\ud_{C,x}f(\varphi(y)-\varphi(x))+o(d(x,y)) \quad \textrm{as}\quad  y\to x.
\end{align}

Not every space admits Lipschitz differentiable structure, as shown by the so called \emph{Rickman's rug} $X:=[0,1]^2$ equipped with the metric $d((x_1,y_1), (x_2,y_2))=|x_1-x_2|+|y_1-y_2|^\alpha$, where $\alpha \in (0,1)$, and $\mu=\mathcal L^2|_X$. Indeed, a Weierstrass-type function in the $y$-variable combined with \cite[Theorem 1.14]{sch16b} would yield non-horizontal rectifiable curves if the space were a differentiability space, contradicting the fact that all rectifiable curves in $X$ are horizontal.

Here, we introduce \emph{$p$-weak differentiable structures}, which exist in much more generality (including Rickman's rug, see the discussion after Definition \ref{def:pweakchart}), adapting Cheeger's construction by substituting \eqref{eq:cheegerdiff} for a weaker curve-wise control. To accomplish this, we replace the pointwise Lipschitz constant by the minimal $p$-weak upper gradient in the definition of ``infinitesimal linear independence'' \eqref{eq:weakcha} and use Theorem \ref{thm:curvewise-p=1} to circumvent the difficulties arising from the fact that the latter is defined only up to a null-set.

In the remainder of the introduction, we use the notation $|Df|_p$ for the minimal $p$-weak upper gradient of $f\in N^{1,p}_{loc}(X)$ and refer to Section \ref{sec:prelims} for more discussion on this notation. Given $p\ge 1$ and $N\in\N$, we say that a Sobolev map $\varphi\in N_{loc}^{1,p}(X;\R^N)$ is \emph{$p$-independent} in $U\subset X$, if
\begin{align}\label{eq:weakcha}
\underset{v\in S^{N-1}}\essinf \, |D(v\cdot\varphi)|_p>0\quad\mu\textrm{-a.e. in }U,
\end{align}
and \emph{$p$-maximal} in $U$, if no \emph{Lipschitz} map into a higher dimensional target is $p$-independent in a positive measure subset of $U$.  Here, we use the essential infimum of an uncountable collection, which agrees $\mu$-a.e. with the pointwise infimum over any countable dense collection of $S^{N-1}$, see Appendix \ref{subsec:esssup}. Note that $p$-maximality does not depend on the particular map $\varphi$ but rather the dimension of its target space.

\begin{defn}\label{def:pweakchart}
	An $N$-dimensional $p$-weak chart $(U,\varphi)$ of $X$ consists of a Borel set $U\subset X$ with positive measure and a Lipschitz function $\varphi:X\to \R^N$ which is $p$-independent and $p$-maximal in $U$. We say that $X$ admits a $p$-weak differentiable structure if it can be covered up to a null set by countably many $p$-weak charts. 
\end{defn}

By convention, zero dimensional $p$-weak charts satisfy $\varphi\equiv 0$ and \eqref{eq:weakcha} is a vacuous condition, while maximality means that $|D f|_p = 0$ $\mu$-a.e on $U$ for every $f\in \LIP_b(X)$ (see also Proposition \ref{prop:zerochart}).  In Section \ref{sec:sobchart} we briefly discuss a lower regularity requirement in Definition \ref{def:pweakchart} and the fact the resulting notion yields essentially the same $p$-weak differentiable structure. We also show that and $N$-dimensional $p$-weak chart $(U,\varphi)$ satisfies $N\le \dim_HU$, where $\dim_HU$ denotes the Hausdorff dimension of $U$, see Proposition \ref{prop:chahausd}. In particular, we have the following theorem. 
	
\begin{thm}
A metric measure space of finite Hausdorff dimension admits a $p$-weak differentiable structure for any $p\ge 1$.
\end{thm}
We refer to Proposition \ref{prop:exists-p-weak-diff-str} for a more technical statement, which immediately implies the theorem. Next, we give an analogue of the Cheeger differential \eqref{eq:cheegerdiff} using $p$-weak charts.


\begin{defn}\label{def:p-weakdiff}
Given an $N$-dimensional $p$-weak chart $(U,\varphi)$ of $X$, a Newtonian function $f\in N^{1,p}(X)$ has a \emph{$p$-weak differential}  $\ud f:U\to (\R^N)^*$ with respect to $(U,\varphi)$, whose values at $x\in U$ are denoted $\ud_x f$, such that
\begin{align}\label{eq:differential}
f(\gamma_s)-f(\gamma_t) = \ud_{\gamma_t} f(\varphi(\gamma_s)-\varphi(\gamma_t)) + o(|t-s|) \quad \text{a.e. } \quad t\in \gamma\inv(U), \text{ as } s \to t
\end{align}
for $p$-a.e. absolutely continuous curve $\gamma$ in $X$. The map $ \ud f$ is called a \emph{$p$-weak differential} (for $f$).
\end{defn}


If the curve $\gamma$ does not enter $U$, or only spends zero length in the set, then condition \eqref{eq:differential} becomes vacuously satisfied with both sides vanishing. 
The $p$-weak differential is uniquely determined almost everywhere equivalence by \eqref{eq:differential}, cf. Lemma \ref{lem:properties}. Further, it is also local, i.e. if $g\in N^{1,p}(X)$ and $f|_A=g|_A$ on a positive measure subset $A \subset U$, then $\ud f|_A = \ud g|_A$. The differential satisfies natural computation rules, see Propositions \ref{prop:calcrules} and \ref{prop:sectioncalcrules}.

\begin{thm}\label{thm:differentiability} Suppose $p\ge 1$, and $\varphi: X \to \R^N$ is a $p$-weak chart on $U$. Then any $f\in N^{1,p}(X)$ has a $p$-weak differential $\ud f:U\to (\R^N)^*$ with respect to $\varphi$, and the map $f\mapsto \ud f$ is linear.
	
	Moreover, for $\mu$-a.e. $x\in U$, there is a norm $|\cdot|_x$ on $(\R^N)^*$, such that $x\mapsto |\xi|_x$ is Borel for every $\xi\in (\R^N)^*$, and
	$$|\ud f|_x = |Df|_p(x)\quad \mu\textrm{-a.e. }x\in X,$$
	for every $f\in N^{1,p}(X)$.
\end{thm}

 Whereas Lipschitz functions are differentiable with respect to Cheeger charts, \eqref{eq:weakcha} yields only the curvewise control \eqref{eq:differential}. Indeed, if there are very few or no rectifiable curves, or if the curves only point into certain directions, then the $p$-weak differential vanishes, or measures only these directions, respectively.  For example, given a fat Cantor set $K\subset \R^n$ with $\mathcal L^n(K)>0$,  $X:=(K,d_{Eucl},\mathcal L^n|_K)$ is a Lipschitz differentiability space but the minimal weak upper gradient of every Lipshitz function is zero. 
 On the other hand, Rickman's rug admits non-trivial $p$-weak charts $\varphi(x,y)=x$.  The $p$-weak differential in this case can be identified with the $x$-derivative,  $\ud f \equiv \partial_x f$, and the only curves with positive $p$-modulus are those which are horizontal. These examples demonstrate that $p$-weak differentiable structures might exist for spaces not admitting a Cheeger structure, but the two need not coincide even if both exist.  However, if a Poincar\'e inequality is present, the two structures coincide.

\begin{thm}\label{thm:PI}
	Suppose $X$ is a $p$-PI space for $p\ge 1$. Then any $p$-weak chart $(U,\varphi)$ of $X$ is a Cheeger chart.	
\end{thm}
It follows from the discussion after Definition \ref{def:pweakchart} that a $p$-PI space admits $p$-weak charts. In Subsection \ref{sec:intro-gigli-lip}, we obtain a precise statement on the relationship between the $p$-weak and Lipschitz differentiable structure, as well as a characterization of the existence  of $p$-weak differentiable structures in terms of Gigli's cotangent module, cf. \cite{gig18}. Here we mention a noteworthy corollary of the existence of a $p$-weak differentiable structure. 

\begin{thm}\label{thm:lipdense}
	Let $p\ge 1$. If $X$ admits a $p$-weak differential structure, then $\LIP_b(X)$ is norm-dense in $\Ne pX$.
\end{thm}

Theorem \ref{thm:lipdense} has been obtained by other methods for $p>1$ in \cite{ambgigsav} but is new in the case $p=1$. In particular, we highlight that the density holds if $X$ has finite Hausdorff dimension. 

\subsection{Connections to Cheeger's and Gigli's differentiable structures}\label{sec:intro-gigli-lip} Together with the pointwise norm from Theorem \ref{thm:differentiability}, a $p$-weak differentiable structure gives rise to a \emph{$p$-weak cotangent bundle $T^*_pX$} over $X$, analogous to the measurable $L^\infty$-cotangent bundle $T_C^*X$ arising from the Lipschitz differentiable structure \cite{che99,kei04}, which is equipped with the pointwise norm
$$
|\xi|_{C,x}:=\Lip(\xi\circ\varphi)(x),\quad \xi\in (\R^N)^*
$$
for $\mu$-a.e. $x\in U$, where $(U,\varphi)$ is an $N$-dimensional Cheeger chart. For any $f\in \LIP_b(X)$, the differentials $\ud f$ and $\ud_Cf$ are sections of the cotangent bundles $T_p^*X$ and $T_C^*X$, respectively. We refer to Section \ref{sec:differentialstruct} for the precise definition of measurable $L^\infty$-bundles and their sections. 

In the next theorem we show that there is a submetric bundle map $T_C^*X\to T_p^*X$, and give a condition under which the bundle map is an isometric isomorphism. See Section \ref{sec:differentialstruct} for the definition of bundle maps. In the statement, a modulus of continuity is an increasing continuous function $\omega:[0,\infty)\to [0,\infty)$ with $\omega(0)=0$, and a linear submetry between normed spaces $V$ and $W$ is a surjective linear map $L:V\to W$ with $L(B_V(r))=B_W(r)$.

\begin{thm}\label{thm:bundles}
	Suppose $X$ admits a Cheeger structure and let $p\ge 1$. There is a bundle map $\pi=\pi_{C,p}:T^*_CX\to T^*_pX$ which is a linear submetry $\mu$-a.e. and satisfies 
	\begin{align}\label{eq:ch-pweak-compatible}
	\pi_x(\ud_{C,x}f)=\ud_{x}f,\quad \mu-\textrm{a.e. }x\in X
	\end{align}
	for every $f\in \LIP_b(X)$. If there exists a collection $\{\omega_x\}_{x\in X}$ of moduli of continuity satisfying
	\[
	\Lip f(x)\le \omega_x(|Df|_p(x))\quad  \mu-\textrm{a.e. on }X
	\]
	for every $f\in \LIP_b(X)$, then $\pi_{C,p}$ is an isometric bijection $\mu$-a.e.
\end{thm}
Theorem \ref{thm:bundles} follows from \cite[Theorem 1.1]{tet2020} and the following theorem, which identifies the space $\Gamma_p(T_p^*X)$ of $p$-integrable sections of the $p$-weak cotangent bundle $T_p^*X$ with Gigli's cotangent module $L^p(T^*X)$. We refer to Section \ref{sec:connection-gigli-cheeger} for the relevant definitions, and remark here that Gigli's construction is the most general in the sense that $L^p(T^*X)$ can be defined for any metric measure space. It is a priori defined only as an abstract $L^p$-normed $L^\infty$-module in the sense of \cite{gig15,gig18}.

We say that $L^p(T^*X)$ is \emph{locally finitely generated} if $X$ has a countable Borel partition $\mathscr B$ so that each $B\in \mathscr B$ admits a finite generating set in $B$. Here, a collection $\mathscr V\subset L^p(T^*X)$ is a generating set in $B$ (or generates $L^p(T^*X)$ in $B$) if $\chi_BL^p(T^*X)$ is the smallest closed submodule of $L^p(T^*X)$ containing $\chi_Bv$ for every $v\in \mathscr V$. Gigli's cotangent modules admit a \emph{dimensional decomposition}, i.e. a Borel partition $\{A_N\}_{N\in\N\cup\{\infty\}}$ of $X$ so that $L^p(T^*X)$ admits a generating set of cardinality $N$ (and no smaller) in $A_N$, for each $N$. For $N=\infty$, no finite set generates $L^p(T^*X)$ in $A_N$. The dimensional decomposition is uniquely determined up to $\mu$-negligible sets.

In the statement below, we denote by $\ud_Gf$ and $|\cdot|_G$ the abstract differential and pointwise norm in the sense of Gigli, see Theorem \ref{thm:cotangmodule}. A morphism between $L^p$-normed $L^\infty$-modules (i.e. a continuous $L^\infty$-linear map) is said to be an isometric isomorphism if it preserves the pointwise norm, and has an inverse that is a morphism.



\begin{thm}\label{thm:gigli}
	Let $X$ be a metric measure space and $p\ge 1$. Then $X$ admits a $p$-weak differentiable structure if and only if $L^p(T^*X)$ is locally finitely generated. In this case, 
	\begin{itemize}
		\item[(a)] there exists an isometric isomorphism $\iota: \Gamma_p(T^*_pX)\to L^p(T^*X)$ of normed modules, satisfying $\iota(\ud f)=\ud_G f$ for every $f\in N^{1,p}(X)$ and uniquely determined by this property, 
		\item[(b)] each set $A_N$ in the dimensional decomposition of $X$ can be covered up to a null-set by $N$-dimensional $p$-weak charts, 
		\item[(c)] $N\le \dim_H(A_N)$ for each $N\in\N$.
	\end{itemize}
\end{thm}
Theorem \ref{thm:gigli} gives a concrete interpretation of Gigli's cotangent module, and bounds the Hausdorff dimension of the sets in the dimensional decomposition. As corollaries we obtain the reflexivity of $N^{1,p}(X)$ when $p>1$, and a characterization of infinitesimal Hilbertianity in terms of the pointwise norm of Theorem \ref{thm:differentiability} when $p=2$, for spaces admitting a $p$-weak differentiable structure, see Corollary \ref{cor:giglicor}.  Reflexivity could also be obtained directly from Theorem \ref{thm:differentiability} following the argument in \cite[Section 4]{che99}.

\subsection{Acknowledgements}

The first author was partially supported by National Science Foundation under Grant No. DMS-1704215 and by the Finnish Academy under Research postdoctoral Grant No. 330048.  The second author was supported by the Swiss National Science Foundation Grant 182423. Throughout the project the authors have had insightful discussions with Nageswari Shanmugalingam, which have been tremendously useful. A further thanks goes to Jeff Cheeger for helpful comments and inspiration for the project. The authors thank IMPAN for hosting the semester ``Geometry and analysis in function and mapping theory on Euclidean and metric measure space'' where this research was started. Through this workshop the authors were partially supported by grant \#346300 for IMPAN from the Simons Foundation and the matching 2015-2019 Polish MNiSW fund. The first author additionally thanks University of Jyv\"askyl\"a, where much of this work was done.

\section{Preliminaries}\label{sec:prelims}

Throughout this paper $X=(X,d,\mu)$ will be a complete separable metric measure space equipped with a Radon measure $\mu$ finite on balls. We denote by $C(I;X)$ the space of continuous curves $\gamma\colon I \to X$ equipped with the metric of uniform convergence, and by $AC(I;X)$ the subset of absolutely continuous curves in $X$. Mostly, we will be concerned with statements independent of parametrization, thus the choice of $I$ is immaterial. However, when we need to refer to the end points of the curve, then we will take $I=[0,1]$.

If $\gamma$ is a curve, its value at $t\in I$ is denoted by $\gamma_t\defeq \gamma(t)$. If $f:X \to \R^N$ is a function, we also use this notation as $(f\circ\gamma)_t=f(\gamma_t)$. The derivative of $f$ in the direction of $\gamma$ at $\gamma_t$, when it exists, is denoted $(f\circ\gamma)_t' = (f\circ\gamma)'(t)$. The metric derivative of the curve, in the sense of say \cite[Section I.1]{AGS08}, is defined as $|\gamma_t'| = \lim_{h\to 0}\frac{d(\gamma_{t+h},\gamma_t)}{h}$, when it exists. 
The metric derivative is defined almost everywhere on $I$ for $\gamma\in AC(I;X)$.

\subsection{Plans and modulus}
A finite measure $\bm\eta$ on $C(I;X)$ is called a \emph{plan} if it is concentrated on $AC(I;X)$, and a \emph{$q$-plan} if the \emph{barycenter} $\ud \bm\eta^\#:=e_*(|\gamma_t'|\ud t\ud \bm\eta)$ satisfies $\ud \bm\eta^\#=\rho\ud\mu$ for some $\rho\in L^q(\mu)$. We denote by $AC_q(I;X)$ the space of curves $\gamma\in AC(I;X)$ satisfying $\displaystyle \int_0^1|\gamma_t'|^q\ud t<\infty$, and say that a $q$-plan $\bm\eta\in \mathcal P(C(I;X))$ is a \emph{$q$-test plan}, if it is concentrated on $AC_q(I;X)$ and
\begin{align*}
e_{t*}\bm\eta\le C\mu\textrm{ for every $t\in I$, and }\int\int_0^1|\gamma_t'|^q\ud t\ud\bm\eta<\infty
\end{align*}
for some constant $C>0$. Here $e_t:C(I;X)\to X$ is the map $e_t(\gamma)=\gamma_t$.

\begin{remark} Every $q$-test plan is a also a $q$-plan. However, the converse can fail for two reasons. A $q$-test plan fixes a given parametrization for curves (with an integrability condition on the speed), and insists on a compression bound $e_{t*}(\bm\eta)\leq C\mu$. However, for each $q$-plan supported on $\Gamma\subset AC(I;X)$, one can construct associated $q$-test plans supported on reparametrized curves, which are subcurves of curves in $\Gamma$. 

The argument for this is a combination of two observations in \cite{amb13}. First, for each $q$-plan one can re-parametrize curves to get a plan with a good "parametric barycenter" \cite[Definition 8.1 and Theorem 8.3]{amb13}. The parametric barycenter depends on the parametrization, while the barycenter $\bm\eta^\#$ does not. The second point concerns the compression bound, where given the previous plan, one can take sub-segments of curves and average these over shifts to get a compression bound, which is explained as part of the proof  of \cite[Theorem 9.4]{amb13}.

 This remark would allow, for example, to phrase Theorem \ref{thm:curvewise-p=1} with test plans instead of plans, if one were so inclined. 
\end{remark}

If $\Gamma \subset C(I;X)$ is a family of curves, then a Borel function  $\rho \colon X \to [0,\infty]$ is called admissible, if $\int_\gamma \rho \ud s \geq 1$ for each rectifiable $\gamma \in \Gamma$. In particular, if there are no rectifiable curves, then this condition is vacuous. We define, for $p \in [1,\infty)$
$$\Mod_p(\Gamma) = \inf_{\rho} \int_X \rho^p \ud \mu,$$
where the infimum is over all admissible $\rho$. We remark, that due to Vitali-Caratheodory, such an infimum can always be taken with respect to lower semi-continuous functions. Notice that the modulus is supported on rectifiable curves and is independent of the parameterization of such curves. We say that a property holds for \emph{$p$-almost every curve} if there is a family of curves $\Gamma_B$ so that $\Mod_p(\Gamma_B)=0$, and the property holds for all $\gamma \in C(I;X) \setminus \Gamma_B$. Modulus is invariant of the parametrization of curves, but some of our statements depend on a parametrization. In those cases, we will say that the property holds for \emph{$p$-almost every absolutely continuous curve in $X$} (or $p$ a.e. $\gamma\in AC(I;X)$) to emphasize that the property holds for each $\gamma \in AC(I;X) \setminus \Gamma_B$ with $\Mod_p(\Gamma_B)=0$. The reader may consult \cite[Sections 4--7]{HKST07} for a more in-depth treatment of modulus, upper gradients and Vitali-Caratheodory.

\begin{remark}\label{rmk:modnull} A crucial fact we will use is that if $\Gamma$ satisfies $\Mod_p(\Gamma)=0$, then for any $q$-plan $\bm\eta$ we have $\bm\eta(\Gamma)=0$ (which holds for $p\in [1,\infty)$ and $q$ its dual exponent). The converse is also true for $p \in (1,\infty)$. See the arguments and discussion in \cite[Sections 4,9]{amb13}. One point here is that if we used $q$-test plans, this relationship would be more complex, and we would need to consider "stable" families of curves, see \cite[Theorem 9.4]{amb13}. The case of $p=1$ is also somewhat subtle, and we will deal with a special case of this issue in Section \ref{sec:curvewise}. The argument of Proposition \ref{prop:duality} would give the converse for compact families of curves and $p=1$. See also, \cite{exnerova2019plans} for a much more detailed exploration of this borderline case.
\end{remark}

The previous remark concerns an inequality relating modulus and $q$-plans. However, there is a closer connection, and in a sense these are dual to each other. Previously, this has been explored in \cite[Theorem 5.1]{amb13}  for $p>1$, and in \cite[Theorem 6.3]{exnerova2019plans} for $p=1$. Due to its importance for us, we summarize one main consequence of these results. We further briefly describe the main steps of a direct proof from \cite[Proposition 4.5]{davideb20}. A similar argument appeared previously in a more specific context in \cite[Theorem 3.7]{shanmunsemmes}.

\begin{prop}\label{prop:duality} Let $p \in [1,\infty)$ and $q$ its dual exponent with $p\inv+q\inv = 1$. If $K \subset C(I;X)$ is a compact family of curves, and  $\Mod_p(K) \in (0,\infty)$, then there exists a $q$-plan $\bm \eta$ with  $\spt(\bm\eta)\subset K$. 
\end{prop}
\begin{proof}
A power of the modulus $\Mod_p(K)^{1/p}$ arises from a convex optimization problem on $\rho$ with a constraint for every curve $\gamma\in K$. A dual formulation of this corresponds to a variable for each constraint, i.e. a measure $\nu$ supported on $K$. Thus, it is reasonable to consider a modified Lagrangian defined by
$$\Phi(\rho,\nu) = ||\rho||_{L^p}-\Mod_p(K)^{1/p} \int_K \int_\gamma \rho ~\ud s\ud \nu_\gamma,$$
where $\rho:X \to [0,\infty]$ is a function and $\nu$ is a probability measure supported on $K$. Let $P(K)$ be the collection of these probability measures supported on $K$ equipped with the topology of weak* convergence. In order to obtain required continuity, we will restrict to $\rho \in G$ with $G\defeq \{\rho:X \to [0,1], \rho \text{ compactly supported and  continuous in }X\}$. The set $G$ is equipped with the topology of uniform convergence. Then $\Phi:G\times P(K) \to \R$ is a functional with two properties: a) $\Phi(\cdot, \nu)$ is convex and continuous for each $\nu \in P(K)$, b) $\Phi(\rho,\cdot)$ is concave and upper semi-continuous for each  $\rho:X \to [0,1]$. Further $P(K)$ is compact and convex in the weak* topology and $G$ is a convex subset. 

By Sion's minimax theorem, see e.g. statement in \cite[Theorem 4.7]{davideb20}, we have
$$\sup_{\nu \in P(K)}\inf_{\rho \in G} G(\rho,\nu) = \inf_{\rho \in G} \sup_{\nu \in P(K)}G(\rho,\nu).$$

We can compute $\inf_{\rho \in G} \sup_{\nu \in P(K)}G(\rho,\nu)\geq 0$. Indeed, given any $\rho \in G$, we can use the definition of modulus to find a $\gamma\in K$ with $\int_\gamma \rho \ud s \leq \frac{||g||_p}{\Mod_p(K)^{-1/p}}$. If we choose $\nu=\delta_\gamma$, a Dirac measure on $\gamma$, the bound immediately follows.

Therefore, we have also $\sup_{\nu \in P(K)}\inf_{\rho \in G} G(\rho,\nu)\geq 0$. But, up to showing that this supremum is attained, there must be some $\bm\eta\in P(K)$ for which we get $\inf_{\rho \in G} G(\rho,\bm\eta)\geq 0$. After unwinding the definition of a $q$-plan, and an application of Radon-Nikodym on $X$, the measure $\bm\eta$ is our desired $q$-plan. 
\end{proof}


\subsection{Sobolev spaces and functions}
A function $f: (X,d_X) \to (Y,d_Y)$ between two metric spaces is called Lipschitz if $\LIP(f)\defeq \sup_{x,y \in X, x\neq y}\frac{d_Y(f(x),f(y))}{d_X(x,y)}<\infty$. A bijection $f:X\to Y$ is called Bi-Lipschitz if $f$ and $f\inv$ are Lipschtiz. Further, if $x \in X$, we define the local Lipschitz constant as
$$\Lip f(x) \defeq \limsup_{y\to x, y\neq x} \frac{d_Y(f(x),f(y))}{d_X(x,y)}.$$
Let $\LIP_b(X)$ be the collection of Lipschitz maps $f:X\to \R$ with bounded support.
\begin{defn}\label{def:ugalongpath}
Let $f:X\to \R\cup\{\pm\infty\}$ be measurable, $g:X\to [0,\infty]$ a Borel function, and $\gamma:I\to X$ a rectifiable path. We say that $g$ is an upper gradient of $f$ along a rectifiable curve $\gamma:[0,1]\to X$, if $\int_\gamma g ~\ud s<\infty$ and
\[
|f(\gamma_t)-f(\gamma_s)|\le \int_{\gamma|_{[s,t]}}g,
\]
for each $s<t$ with $s,t\in I$ with the convention $\infty-\infty=\infty$. We say that $g$ is an upper gradient of $f$, if it is an upper gradient along every rectifiable curve, and a $p$-weak upper gradient if $g$ is an upper gradient of $f$ along $p$-a.e. rectifiable curve.
\end{defn}

The space $N^{1,p}(X)$ is defined as all $\mu$-measurable functions  $f\in L^p(X)$ which have an upper gradient $g$ in $L^p(X)$. The (semi-)norm on this space is defined as $$\|f\|_{N^{1,p}}=\left( \|f\|_{L^p}^p+\inf\|g\|_{L^p}^p \right)^{1/p},$$ where the infimum is taken over all $L^p$-integrable upper gradients $g$ of $f$. The theory of these spaces was largely developed in \cite{sha00}, see also \cite{HKST07} for most of the classical theory. By the results there combined with an observation of Haj{\l}asz \cite{haj03} in the case of $p=1$,  one can show that there always exists a unique minimal $g_f$, which is an upper gradient along $p$-almost every path, and for which $\|f\|_{N^{1,p}}=\left(\|f\|_{L^p}^p+\|g_f\|_{L^p}^p \right)^{1/p}$. We call $g_f$ the \emph{minimal $p$-upper gradient}. 
Similarly, we can define $f\in N^{1,p}_{loc}(X)$ if $f\eta \in N^{1,p}$ whenever $\eta\in \LIP_b(X)$. In these cases we also can define a minimal $p$-upper gradient  $g_f$, so that $\eta g_f \in L^p(X)$ for every $\eta\in \LIP_b(X)$. In other words, $g_f \in L^p_{loc}(X)$.

We denote by $N^{1,p}(X;\R^N)\simeq \Ne pX^N$ the space of functions $\varphi\colon X\to\R^N$ so that each component is in $N^{1,p}$. Similarly, we define $\LIP_b(X;\R^N)\simeq \LIP_b(X)^N$.

Another notion of Sobolev space can be defined using \emph{$q$-test plans} and we denote it $W^{1,p}(X)$, with $|D f|_p$ denoting the minimal gradient of $f\in W^{1,p}(X)$. Namely, a function $f\in L^p(\mu)$ belongs to the Sobolev space $W^{1,p}(X)$ if there exists $g\in L^p(\mu)$ such that
\[
\int|f(\gamma_1)-f(\gamma_0)|\ud\bm\eta\le \int\int_0^1g(\gamma_t)|\gamma_t'|\ud t\ud\bm\eta
\]
for every $q$-test plan $\bm\eta$ on $X$. The space has a norm $\|f\|_{W^{1,p}}=\left(\|f\|_{L^p}^p + \inf_g \|g\|_{L^p}^p\right)^{1/p}$, where the infimum is over all such functions $g$. We refer to \cite{dimarino19} for details.

Note that any representative of an element of $W^{1,p}(X)$ still belongs to $W^{1,p}(X)$, whilst a representative of an element in $\Ne pX$ belongs to $\Ne pX$ if and only if they agree outside a $p$-exceptional set. The next theorem says that up to this ambiguity of representatives, the two approaches produce the same object. The measurability conclusion is also a corollary of \cite{seb2020}. We refer to \cite{amb15} for a proof.
\begin{thm}\label{thm:W=N}
	Let $p\in (1,\infty)$. If $f\in \Ne pX$, then $f\in W^{1,p}(X)$ and $g_f=|Df|_p$ $\mu$-a.e. Conversely, if $f\in W^{1,p}(X)$, then $f$ has a Borel representative $\bar f\in \Ne pX$ with $g_{\bar f}=|Df|_p$ $\mu$-a.e.
\end{thm}

\section{Curvewise (almost) optimality of minimal upper gradients}\label{sec:curvewise}

\subsection{Upper gradients with respect to plans}

Given a plan $\bm\eta$, we can speak of a gradient along its curves.

\begin{defn} If ${\bm \eta}$ is a $q$-plan, and $f \in N^{1,p}$ then a Borel function $g$ is an ${\bm \eta}$-upper gradient if for ${\bm \eta}$-almost every $\gamma$, $g$ is an upper gradient of $f$ along $\gamma$. 
\end{defn}

The following lemma gives a notion of a minimal $\bm\eta$-upper gradient and shows how to compute it by using derivatives along curves.

\begin{lemma}\label{lem:uppergradient} Suppose $g_f$ is a minimal upper gradient, and $\bm \eta$ is any $q$-plan and $\ud \bm\pi = \ud\bm\eta |\gamma_t'|\ud t$ with disintegration $\bm\pi_x,$ then
	\begin{enumerate}
		\item $g_{\bm \eta} = ||(f \circ \gamma)_t'/\gamma_t'||_{L^{\infty}(\bm \pi_x)}$ is a $\bm \eta$-upper gradient.
		\item $g_{\bm \eta} \leq g$ for any other $\eta$-upper gradient for almost every $x \in X$.
		\item $g_{\bm \eta} \leq g_f$ for almost every $x \in X$.
		\item Suppose $\bm \eta'$ is another $q$-plan, and $\bm\eta \ll \bm\eta'$, then $g_{\bm \eta} \leq g_{\bm \eta'}$.
	\end{enumerate}

\end{lemma}

\begin{proof}
	Let $g_f$ be the minimal $p$-upper gradient for $f$.  By Lemma \ref{lem:nullborel} there is a Borel family $\Gamma_0 \subset C(I;X)$, so that $f$ is absolutely continuous on each curve $\gamma \not\in \Gamma_0^c$ with upper gradient $g_f$ and so that $\bm\eta(\Gamma_0)=0$. By Corollary \ref{cor:borel} and Lemma \ref{lem:measurable-integral} there is a set $N \subset C(I;X) \times I$ so that for $\bm\pi(N)=0$, and for each $(\gamma,t) \not\in N$ both $(f\circ\gamma)'(t)$ and $|\gamma_t'|$ are defined and measurable. Let $M_0 = \Gamma_0 \times I \cup N$, we get $\bm\pi(M_0)=0$. For each curve $\gamma \not\in \Gamma_0$ the function $f$ is absolutely continuous with upper gradient $\frac{(f\circ\gamma)_t'}{|\gamma_t'|}$. Since $g_{\bm \eta}(\gamma_t)\geq \frac{(f\circ\gamma)_t'}{|\gamma_t'|}$ for $\bm\pi$-almost every $(\gamma,t) \in M_0$, we have that $g_{\bm\eta}$ is an $\bm\eta$ upper gradient.
	
	
	
	
	If $g$ is any other Borel $\bm \eta$-upper gradient, then the set of $(\gamma,t)\in \Diff(f)\setminus M_0$ with $(f \circ \gamma)'_t/|\gamma_t'| > g(\gamma_t)$ must have null measure, and thus the claim follows by Fubini and the definition in $(1)$.
	
	The function $g_f$ is an upper gradient for $f$ on curves in $\Gamma_0^c$, and thus the claim follows again from curvewise absolute continuity and by showing that  the set of $(\gamma,t)$ with $(f \circ \gamma)_t'/|\gamma_t'| > g_f(\gamma_t)$ must have null $\bm\pi$-measure.
	The final claim follows since $g_{\bm \eta'}$ must be a $\bm \eta$-upper gradient for $f$. 
	
\end{proof}

\subsection{Proof of Theorem \ref{thm:curvewise-p=1}} In this subsection we prove Theorem \ref{thm:curvewise-p=1}. The idea is that for each $q$-plan $\bm \eta$ we can associate a gradient "along" the curves of such a plan. Each such  gradient must be less than the minimal upper gradient, and thus the task is to show that by varying over different plans $\bm \eta$ we can obtain the minimal upper gradient through maximization. In order to show equality of the result of this maximization, we argue by contradiction, that if it were not a minimal upper gradient, then we could witness this by a given plan.  This is the core of the following result. It should be compared to \cite[Sections 9--11]{amb13}, where a similar analysis is done, but with different terminology and only for $p>1$.  In the following statement we will need to refer to end points of curves, and thus choose the domain of curves as $I=[0,1]$. 

\begin{lemma}\label{lem:lessthangradient}
Let $p \in [1,\infty)$, and $q$ be its dual exponent. Let $f \in N^{1,p}(X)$. Suppose $g$ is any non-negative Borel function so that $A=\{g<g_f\}$ has positive measure, then there exists a $q$-plan $\bm\eta$, so that for $\bm\eta$-almost every curve $\gamma:[0,1] \to X$ we have

\begin{equation}\label{eq:baddir}
|f(\gamma_1)-f(\gamma_0)| > \int_\gamma g \ud s.
\end{equation}
\end{lemma}

\begin{proof}  By by Vitali-Caratheodory we may find a lower semicontinuous $\tilde{g} \geq g$ which is integrable and so that 
$\tilde{A}=\{\tilde{g}<g_f\}$ has positive measure. We will suppress the tildes below to simplify notation and thus only consider the case of $g$ lower semicontinuous. Since $g<g_f$ on a positive measure subset, $g$ cannot be a minimal upper gradient, and thus there must exist a family $\Gamma \subset C(I;X)$ of curves with $\Mod_p(\Gamma)>0$, so that Estimate \eqref{eq:baddir} holds for each $\gamma \in \Gamma$. Modulus is invariant under reparametrization of curves and so we may only the subset of those $\gamma \in \Gamma$ which are Lipschitz.  We want to find a plan supported on $\Gamma$. However, the issue with this is that since $p=1$ is allowed the family $\Gamma$ may not be compact, the duality of modulus and $q$-plans may fail. So, we seek to ``cover'' $\Gamma$, up to a null modulus family by compact families. This covering is done in an iterative way.

Fix an $R$ so that the modulus of $\Gamma_R$ of those curves in $\Gamma$, which are contained in a ball $B(x_0, R)$ for some fixed $x_0\in X$, is positive. Since $f$ is measurable and $X$ is complete and separable, Egorov's theorem implies the existence of an increasing sequence of compact sets $K_n$ satisfying $\mu(B(x,R) \setminus \bigcup K_n)=0$ for which $f|_{K_n}$ is continuous for each $n$.  Define $\mu(B(x,R) \setminus K_n) = \epsilon_n$. By passing to a sub-sequence of $n$, we may assume that $\sum_n \sqrt{\epsilon_n} <1$.
	
	Define $\overline{\Gamma}$ to be the collection of $\gamma \in \Gamma_R$ so that $f$ is absolutely continuous on $\gamma$, and that $\mathcal{H}^1(\gamma \setminus (\bigcup_{n=1}^\infty K_n))=0$. This holds for $\Mod_p$-almost every curve, since $f \in N^{1,p}$ and since $p$-almost every curve spends measure zero in the null set $X \setminus \bigcup_{n=1}^\infty K_n$. Thus, $\Mod_p(\overline{\Gamma})>0$.
	
	Next, let $\Gamma^m$ be those curves $\gamma: I \to X$, which are $m$-Lipschitz, so that $\len(\gamma) \leq m|b-a|, \diam(\gamma) \geq \frac{1}{m}$, $\gamma_0, \gamma_1 \in K_m$ and Estimate \eqref{eq:baddir} holds. We will show that every $\gamma \in \overline{\Gamma}$ contains a subcurve, up to reparametrization, in  $\bigcup_{m=1}^\infty \Gamma^m$. From this, and Lemma \cite[Lemma 1.34]{bjo11}, it follows that $\Mod_p(\bigcup_{m=1}^\infty \Gamma^m)>0$, and thus there is some $M>0$ so that $\Mod_p(\Gamma^M)>0$. It is easy to show that $\Gamma^m$ is a closed family of curves in $C(I;X)$ with respect to uniform convergence, since $g$ is taken to be lower semicontinuous (see e.g. \cite[Proposition 4]{keith03}).
	
	To obtain the previous fact, consider a non-constant curve $\gamma \in \overline{\Gamma}$. We have
	$$|f(\gamma_1)-f(\gamma_0)| > \int_\gamma g ~\ud s.$$
	We may also parametrize $\gamma$ by constant speed as the claim is invariant under reparametrizations.
	
	By parametrization with unit speed, we have that $|I \setminus \bigcup_{n=1}^\infty \gamma^{-1}(K_n)|=0$ and that $f \circ \gamma$ is continuous. Since $\int_\gamma g \ud s < \infty$, and $f \circ \gamma$ is continuous, we can find (for all $n \geq N$ for some $N\in\N$) sequences $a_n,b_n \in [0,1]$ so that $\lim_{n \to \infty} a_n =a$, $\gamma_{a_n} \in K_n$, $\gamma_{b_n} \in K_n$ and $\lim_{n \to \infty} b_n =b$. Then, for sufficiently large $n$
	
	$$|f(\gamma_{b_n})-f(\gamma_{a_n})|> \int_{\gamma|_{[a_n,b_n]}} g \ud s.$$
	
	For $n$ large enough we also have $\len(\gamma_{[a_n,b_n]}) \leq n|b-a|, \diam(\gamma_{[a_n,b_n]})) \geq \frac{1}{n}$. Since the curves are parametrized by unit speed, they are then $n$-Lipschitz. So $\gamma' = \gamma_{[a_n,b_n]}$ is, up to a reparametrization, in $\Gamma^n$ for $n$ large enough, and the claim follows.

	Fix $M>0$ so that $\Mod_p(\Gamma^M)>0$. Next, choose $\delta<\min(\Mod_p(\Gamma^M),1)$. Define $\delta_n = \epsilon_n^{1/2p}$.  Choose $N$ so that $\sum_{n=N}^\infty \sqrt{\epsilon_n} < \delta^{1+p} / 2$.
	
	 Let $\Gamma^M_t$ be the family of curves $\gamma \in \Gamma^M$ so that $\int_{\gamma} 1_{X \setminus K_n} \ud s \leq \delta \delta_n $ for each $n \geq N$. Since $\left(\sum_{n\geq N} \left(\frac{1_{X \setminus K_n}}{\delta \delta_n}\right)^p \right)^{1/p}$ is a function admissible for $\Gamma^M \setminus \Gamma^M_t$, we have
	
	\[
	\Mod_p(\Gamma^M \setminus \Gamma^M_t) \leq \sum_{n \geq N} \frac{\epsilon_n}{\delta_{~	}^p \delta_n^p} < \delta/2.
	\]
	
	Thus, by sub-additivity of modulus, see e.g. \cite[Theorem 1]{fuglede1957},  $\Mod_p(\Gamma^M_t)\geq \Mod_p(\Gamma^M) - \Mod_p(\Gamma^M \setminus \Gamma^M_t)> \delta/2$. By Lemma \ref{lem:compact}, since $\Gamma^M$ is closed, the family $\Gamma^M_t \subset \Gamma^M$ is a compact family of curves in a complete space. 
	Then, by Proposition \ref{prop:duality} there exists a  $q$-plan $\bm \eta$ supported on $\Gamma^M_t$. Each curve $\gamma \in \Gamma^M_t$ satisfies Inequality \eqref{eq:baddir}, and thus the claim follows. 
	
	
	
\end{proof}

For the following proof, recall that if $A,B\subset X$, then $d(A,B) \defeq \inf_{a\in A}\inf_{b\in B} d(a,b)$, and $N_\epsilon(A) \defeq \cup_{a\in A} B(a,\epsilon)$ for $\epsilon>0$.

\begin{lemma}\label{lem:compact} Suppose that $K_n$ are a compact sets, $\eta_n>0$ constants with $\lim_{n \to \infty} \eta_n = 0$, $L>0$ and let $\Gamma \subset C(I;X)$ be a closed family of curves in a complete space $X$. Let $\Gamma^t$ be the family of curves $\gamma\in\Gamma$ for which $\len(\gamma) \leq L$, $\diam(\gamma) \geq \frac{1}{L}$ and which are $L$-Lipschitz with $\int_{\gamma} 1_{X \setminus K_n} \ud s \leq \eta_n$ for each $n\in\N$. Then $\Gamma^t$ is compact.
\end{lemma}
\begin{proof}Let $I=[a,b]$. Since $\Gamma$ and $\Gamma^t$ are closed, it suffices to show pre-compactness. 

Let $\gamma \in \Gamma^t$. We may suppose that $\eta_n < \frac{1}{2L}$ by restricting to large enough $n$. Then, we have for each $n$
$$\int_{\gamma} 1_{K_n} \ud s = \int_{\gamma} 1 \ud s - \int_{\gamma} 1_{X \setminus K_n} \ud s \geq \diam(\gamma) - \eta_n > \frac{1}{L}-\eta_n.$$ 

Thus $\gamma \cap K_n \neq \emptyset.$ Moreover, if $t \in I$, and $d(\gamma_t,K_n)=s$, then there will be a sub-segment of length at least $\min(s,\diam(\gamma_k)/2)$ in $X \setminus K_n$. This gives  $\min(s,\diam(\gamma)/2) \leq \eta_n < \frac{1}{2L}$. This is only possible if $s \leq \eta_n$, since  $\diam(\gamma)/2 \geq \frac{1}{L}$. Indeed, we have $d(\gamma,K_n) \leq \eta_n$.

To run the usual proof of Arzel\`a-Ascoli, since we have equicontinuity with the Lipschitz bound, we only need to show that for each $t\in I$ the set $A_t=\{\gamma_t : \gamma \in \Gamma^t\}$ is pre-compact. However, since $X$ is complete, it suffices to show that $A_t$ is totally bounded. Fix $\epsilon>0$. We concluded that $d(\gamma,K_n) \leq \eta_n$ for all $n\in\N$. 
Thus, we have for some large $n$ that $\eta_n \leq \epsilon/4$ and that $A_t \subset N_{\eta_n}(K_n) \subset N_{\epsilon/4}(K_n)$. Since $K_n$ is compact, it is totally bounded, and the claim follows by covering $K_n$ by finitely many $\epsilon/4$ balls and noting that $\epsilon>0$ is arbitrary. 
\end{proof}

\begin{proof}[Proof of Theorem \ref{thm:curvewise-p=1}] Let $\Pi_q$ be the set of all $q$- plans, and for each $\bm\eta \in \Pi_q$ with its disintegration being given by ${\bm \pi_x}$ define
	
	$$g_{\bm \eta}(x)= \left\| \frac{(f \circ \gamma)'_t}{|\gamma_t'|}\right\|_{L^{\infty}(\bm\pi_x)}.$$
	
	Finally, define
	
	$$|D_\pi f| = \esssup_{\bm \eta \in \Pi_\infty}g_\pi(x).$$
	\vskip.2cm
	\noindent \textbf{Claim 1:} There is a $q$-plan $\tilde{\bm\eta}$ so that $|D_\pi f|=g_{\tilde{\bm\eta}}.$
	\vskip.2cm
 By Lemma \ref{lem:essup}, we can find a sequence $\bm \eta_n$  so that $$g_{\bm\eta_n} \to |D_\pi f|$$ almost everywhere. Consider the measures $\ud\bm\pi^n:=|\gamma_t'|\ud\bm\eta_n\ud t$ on $AC(I;X)\times I$. Set  $a_n=1+ \bm \eta_n(C(I;X)) + ||\frac{d\bm \eta^\#}{d\mu}||_{L^q}+\bm\pi^n(AC(I;X)\times I)$, where $\bm \eta^\#$ is the barycenter of $\bm\eta$, which is absolutely continuous with respect to $\mu$. Let $\tilde{\bm\eta}=\sum_{n=1}^\infty a_n^{-1}2^{-n} \bm\eta_n$. This will be a plan with $g_{\tilde{\bm\eta}} \geq g_{\bm \eta_n}$ for each $n$ by Lemma \ref{lem:lessthangradient}. 
For $\mu$-almost every $x$, we have $g_{\tilde{\bm\eta}} \geq |D_\pi f|$. Then, by Lemma \ref{lem:lessthangradient} we have $\|\frac{(f \circ \gamma)'_t}{|\gamma_t'|}\|_{L^{\infty}(\bm\pi_x)}=|D_\pi f|$, as stated.
	\vskip.2cm
	\noindent \textbf{Claim 2:} We have $|D_\pi f|=g_f$ almost everywhere. 
	\vskip.2cm
	
	Since $g_f$ is an $p$-weak upper gradient, Lemma \ref{lem:lessthangradient} gives $|D_\pi f| \leq g_f$. Suppose for the sake of contradiction then that $|D_\pi f| < g_f$ on a positive measure subset. Then, by Lemma \ref{lem:lessthangradient}, there exists a plan $\bm \eta'$ so that 
	
	$$|f(\gamma_1)-f(\gamma_0)|>\int_\gamma |D_\pi f| \ud s$$
	for $\bm \eta'$-almost every $\gamma$.
	
	However,  by the definition of a plan upper gradient, we have for $\bm \eta'$ almost every curve that
	
	$$|f(\gamma_1)-f(\gamma_0)|\leq \int_\gamma g_{\bm \eta'} \ud s.$$
	
	Now, as $g_{\bm \eta'}\leq |D_\pi f|$ almost everywhere and as $\bm \eta'$ is a $q$-plan, we have for $\bm \eta'$-almost every curve $\gamma$ that 
	$$\int_\gamma g_{\bm \eta'} \ud s \leq \int_\gamma |D_\pi f| \ud s,$$
	which contradicts the above inequalities.
	
	Finally, since $|D_\pi f|=g_{ \tilde{\bm\eta}}=g_f$, then we must have $\mu|_D \ll \tilde{\bm \eta}^\#$. Indeed, otherwise there would be a non-null Borel set $E \subset D$ for which $\mu(E)>0$ and $\tilde{\bm\eta}^{\#}(E)=0$. However, then $g_{\tilde{\bm\eta}}|_E=0$, contradicting the equality $\mu$-almost everywhere.
	
\end{proof}

We now prove Corollary \ref{cor:integral}.

\begin{proof} Let $f \in N^{1,p}$ and consider the plan $\bm \eta'$ obtained from Theorem \ref{thm:curvewise-p=1}. Let $\bm \eta''=r_\ast(\bm\eta')$, where $r:C(I;X) \to C(I;X)$ is the reversal-map which reverses the orientation of every path. Define $\bm \eta = \bm \eta'' + \bm \eta'$. Fix $\epsilon>0$, and define $B=\{(\gamma,t) \in \Diff(f) : g_f(x) \geq \frac{(f\circ \gamma)_t'}{|\gamma_t'|} \geq (1-\epsilon) g_f\}$. 
Since $\left\| \frac{(f\circ \gamma)_t'}{|\gamma_t'|}\right\|_{L^\infty(\bm \pi'_x)}=g_f(x)$ for $\mu$-almost every  $x \in D$, where $\bm \pi'_x$ is the disintegration for $\bm \eta'$,
we have $\bm \pi_x(B)>0$ for $\mu$-almost every $x\in D$ where $\bm \pi_x$ is the disintegration corresponding to $\bm \eta$. Note that, we can remove the absolute values from the supremum norm since for each path $\gamma$ in the support of $\bm \eta'$ we include also its reversal, and $r$ preserves $\bm \eta$. 
\end{proof}

\subsection{Alternative curvewise characterizations of upper gradients when $p>1$}\label{subsec:pgtr1plan}

In this subsection we assume that $p,q\in (1,\infty)$ satisfy $1/p+1/q=1$ and prove a variant Theorem \ref{thm:curvewise-p=1} using test plans representing gradients, introduced by Gigli. 

Given $f\in N^{1,p}(X)$, a $q$-test plan $\bm\eta$ \emph{represents $g_f$}, if
\begin{align*}
\frac{f\circ e_t-f\circ e_0}{t\widetilde E_t^{1/q}}\to g_f\circ e_0\quad \textrm{ and }\quad \widetilde E_t^{1/p}\to g_f\circ e_0\quad \textrm{ in }L^p(\bm\eta),
\end{align*}
where
\begin{align*}
\widetilde E_t(\gamma)=\frac 1t\int_0^t|\gamma_s'|^q\ud s,\quad \gamma\in AC(I;X),\quad \widetilde E_t(\gamma)=+\infty\textrm{ otherwise.}
\end{align*}
A test plan $\bm\eta$ representing the gradient of a Sobolev map $f\in \Ne pX$ is concentrated on ``gradient curves'' of $f$ in an asymptotic and integrated sense. We refer to \cite{gig15, pasqualetto20b} for discussion of the definition we are using here. The following result of Gigli states that Sobolev functions always possess test plans representing their gradient. In the statement, $\mathcal P_q(X)$ denotes probability measures $\nu$ on $X$ with $\int d(x_0,x)^q\ud \nu(x)<\infty$ for some and thus any $x_0\in X$.

\begin{thm}[\cite{gig15}, Theorem 3.14]
	If $f\in \Ne pX$ and $\nu\in \mathcal P_q(X)$ satisfies $\nu\le C\mu$ for some $C>0$, there exists a $q$-test plan $\bm\eta$ \emph{representing} $g_f$, with $e_{0\ast}\bm\eta=\nu$. 
\end{thm}

We now state the main result of this subsection.

\begin{thm}\label{thm:curvewise} Let $f\in \Ne pX$, and $g_f$ be a Borel representative of the minimal $p$-weak upper gradient of $f$, with $D:=\{ g_f>0 \}$ of positive $\mu$-measure. Let $\bm\eta$ be a $q$-test plan representing $g_f$ with $\mu|_D\ll e_{0*}\bm\eta\ll \mu|_D$.
	
	For every $\varepsilon>0$ there exists a Borel set $B\subset \Diff(f)$ such that $\ud\bm\pi:=\chi_B|\gamma_t'|\ud t\ud\bm\eta$ is a positive (finite) measure with $\mu|_D\ll e_*\bm\pi\ll \mu|_D$, whose disintegration $\{\bm\pi_x\}$ with respect to $e$ satisfies
	\begin{align*}
	(1-\varepsilon)g_f(x)\le \frac{(f\circ\gamma)_t'}{|\gamma_t'|}\le g_f(x)\ \textrm{ and }\
	(1-\varepsilon)g_f(x)^{p/q}\le |\gamma_t'|\le (1+\varepsilon)g_f(x)^{p/q}\quad \bm\pi_x-a.e.\ (\gamma,t)
	\end{align*}
	for $\mu$-almost every $x\in D$.
\end{thm}

For the proof, we present the following three elementary lemmas. Denote
\begin{align*}
D_t(\gamma)=\frac{f(\gamma_t)-f(\gamma_0)}{t},\quad \textrm{ and }\quad G_t(\gamma)=\frac 1t\int_0^tg_f(\gamma_s)^p\ud s,\quad \gamma\in AC(I;X),
\end{align*}
and $+\infty$ otherwise. The following observation is essentially made in \cite[Lemma 1.19]{pasqualetto20b} (we are using different notation for our purposes). See Lemma \ref{lem:measurable-integral}(3) for the Borel measurability of the functionals in the  claim.
\begin{lemma}\label{lem:conv}
	Suppose $f\in \Ne pX$ and suppose $\bm\eta$ is a $q$-test plan representing $g_f$. Then
	\begin{align*}
	D_t,G_t,\widetilde E_t\to g_f^p\circ e_0\textrm{ in }L^1(\bm\eta).
	\end{align*}
\end{lemma}
\begin{proof}
	Since $\widetilde E_t^{1/p}\to g_f^p\circ e_0$ in $L^p(\bm\eta)$, it follows that $\widetilde E_t\to g_f^p\circ e_0$ in $L^1(\bm\eta)$. The convergence $D_t\to g_f^p\circ e_0$ is proven in  \cite[Lemma 1.19]{pasqualetto20b}, while $G_t\to g_f^p\circ e_0$ in $L^1(\bm\eta)$ follows from \cite[Proposition 2.11]{gig15}.
\end{proof}

\begin{lemma}\label{lem:youngstable}
	For every $\varepsilon>0$ there exists $\delta>0$ with the following property: if $a,b>0$ and $ \frac{a^p}{p}+\frac{b^q}{q}\le \frac{ab}{1-\delta}$, then $ \left|\frac{a^{p/q}}{b}-1\right|<\varepsilon$.
\end{lemma}
\begin{proof}
	The function $	h:(0,\infty)\to (0,\infty)$, given by $h(t)=\frac 1p t+\frac 1q t^{-q/p}$, has a global minimum at $t=1$, with $h(1)=1$. Thus $h|_{(0,1]}$ and $h|_{[1,\infty)}$ have continuous inverses and it follows that for every $\varepsilon>0$ there exists $\delta>0$ such that if $|1-h(t)|<\delta$ then $|1-t|<\varepsilon$ (expressing the fact that both inverses are continuous at 1). The claim follows from this by noting that if $\frac{a^p}{p}+\frac{b^q}{q}\le \frac{ab}{1-\delta}$ then  $0\le h(t)-1<\delta$ where $t:=a^{p/q}/b$.
\end{proof}
\begin{lemma}\label{lem:posmeas}
	Let $h\le g$ be two integrable functions on $I$ with 
	\begin{align*}
	\liminf_{n\to\infty}\frac{1}{T_n}\int_0^{T_n}g\ud s=:A>0\ \textrm{ and }\ \lim_{n\to\infty}\frac {1}{T_n}\int_0^{T_n}[g-h]\ud s=0.
	\end{align*}
	Then for every $\varepsilon>0$ and $n$, the set $\{(1-\varepsilon)g<h \}\cap [0,T_n]$ has positive $\mathcal L^1$-measure.
\end{lemma}
\begin{proof}
	For large enough $n$ we have that $\displaystyle 0<A/2<\frac{1}{T_n}\int_0^{T_n}g\ud s$ and $\displaystyle 0\le \frac{1}{T_n}\int_0^{T_n}[g-h]\ud s<\varepsilon A/2$. Thus, we may find some $n_0$ for which $\displaystyle \frac{1}{T_n}\int_0^{T_n}[g-h]\ud s<\frac{\varepsilon}{T_n}\int_0^{T_n}g\ud s$, for each $n>n_0. $	It follows that $\displaystyle  \int_0^{T_n}[(1-\varepsilon)g-h]\ud s<0$ for $n>n_0$, and the claim follows from this.
\end{proof}

We will also need the following technical result, compare Lemma \ref{lem:conv}.
\begin{lemma}\label{lem:Econv}
	Let $E\subset X$ be a Borel set, $t>0$, and let
	\begin{align*}
	D_{E,t}(\gamma):=\frac 1t\int_0^t\chi_E(\gamma_s)(f\circ \gamma)_s'\ud s,\quad \gamma\in \Gamma(f). 
	\end{align*}
	Then $D_{E,t}\to (\chi_Eg_f^p)\circ e_0$ in $L^1(\bm\eta)$.
\end{lemma}
\begin{proof}
	Denote
	\[
	F_t(\gamma):=\frac 1t\int_0^tg_f(\gamma_s)|\gamma_s'|\ud s.
	\]
	Since $D_t\le F_t\le \frac 1p G_t+\frac 1q\widetilde E_t$ $\bm\eta$-almost everywhere, Lemma \ref{lem:conv} implies that $F_t\to g_f^p\circ e_0$ and thus $(\chi_E\circ e_0)F_t\to (\chi_Eg_f^p)\circ e_0$ in $L^1(\bm\eta)$. We show that $(\chi_E\circ e_0) F_t-D_{E,t}\to 0$ in $L^1(\bm\eta)$. 
	
	For $\bm\eta$-almost every $\gamma$ we have that
	\begin{align*}
	&|\chi_E(\gamma_0)F_t(\gamma)-D_{E,t}(\gamma)|=\left|\frac 1t\int_0^t [\chi_E(\gamma_0)g_f(\gamma_s)|\gamma_s'|-\chi_E(\gamma_s)(f\circ \gamma)_s']\ud s\right|\\
	\le &  \frac 1t\int_0^t \bigg(|\chi_Eg_f(\gamma_s)-\chi_Eg_f(\gamma_0)||\gamma_s'|+\chi_E(\gamma_0)|g_f(\gamma_s)-g_f(\gamma_0)||\gamma_s'|\,  \\
	&\hspace{8cm} + \chi_E(\gamma_s)[g_f(\gamma_s)|\gamma_s'|-(f\circ\gamma)_s']\bigg)\ud s\\
	\le & \left[\left(\frac 1t\int_0^t|\chi_Eg_f(\gamma_s)-\chi_Eg_f(\gamma_0)|^p\ud s\right)^{1/p}+\left(\frac 1t\int_0^t|g_f(\gamma_s)-g_f(\gamma_0)|^p\ud s\right)^{1/p}\right]\left(\frac 1t\int_0^t|\gamma_s'|^q\ud s\right)^{1/q}\\
	&+F_t(\gamma)-D_t(\gamma).
	\end{align*}
	This estimate, together with the H\"older inequality and Lemma \ref{lem:conv}, yields
	\begin{align*}
	&\limsup_{t\to 0}\int|(\chi_E\circ e_0)F_t-D_{E,t}|\ud\bm\eta\\
	\le & \limsup_{t\to 0}\left[ \left(\int\frac 1t\int_0^t|g_f(\gamma_s)-g_f(\gamma_0)|^p\ud s\ud\bm\eta\right)^{1/p} + \left(\int\frac 1t\int_0^t|\chi_Eg_f(\gamma_s)-\chi_Eg_f(\gamma_0)|^p\ud s\ud\bm\eta\right)^{1/p} \right]\\
	&\times \left(\int g_f^p\circ e_0\ud\bm\eta\right)^{1/q}\\
	= & \limsup_{t\to 0}\left[\left(\frac 1t\int_0^t\|g_f\circ e_s-g_f\circ e_0\|_{L^p(\bm\eta)}^p\ud s\right)^{1/p}+\left(\frac 1t\int_0^t\|\chi_Eg_f\circ e_s-\chi_Eg_f\circ e_0\|_{L^p(\bm\eta)}^p\ud s\right)^{1/p}\right]\\
	&\times \left(\int g_f^p\circ e_0\ud\bm\eta\right)^{1/q}.
	\end{align*}
	Since $s\mapsto h\circ e_s$ is continuous in $L^p(\bm\eta)$ whenever $h\in L^p(\mu)$ (cf. \cite[Proposition 2.1.4]{gigenr}) all terms above tend to zero, proving the claimed convergence.
\end{proof}

\begin{proof}[Proof of Theorem \ref{thm:curvewise}]
	Let $C:=N^c$, where $N$ is as in Corollary \ref{cor:borel}. The function
	\[
	A(\gamma,t)=\frac 1pg_f(\gamma_t)^p+\frac 1q|\gamma_t'|^q, \quad (\gamma,t)\in C,\quad A(\gamma,t)=+\infty,\quad (\gamma,t)\notin C
	\]is Borel. Let $\bm\eta$ represent $g_f$ and satisfy $\mu|_D\ll e_{0*}\bm\eta\ll \mu|_D$. Fix $\varepsilon>0$, let $\delta>0$ be as in Lemma \ref{lem:youngstable}, and set $\delta_0=\min\{\varepsilon,\delta\}$. We define the Borel function 
	\begin{align*}
	H(\gamma,t)=(1-\delta_0)A(\gamma,t)-(f\circ \gamma)_t',\quad (\gamma,t)\in C,\quad H=+\infty\textrm{ otherwise},
	\end{align*}
	cf. Corollary \ref{cor:borel}. The set $B:=\{ H\le 0 \}$ is Borel and, for $(\gamma,t)\in C$, we have
	\begin{align}\label{eq:genericineq}
	(f\circ\gamma)_t'\le g_f(\gamma_t)|\gamma_t'|\le A(\gamma,t).
	\end{align}
	Note that
	\begin{align}\label{eq:conditions}
	H(\gamma,t)\le 0\textrm{ implies } (1-\varepsilon)g_f(\gamma_t)|\gamma_t'|\le (f\circ\gamma)_t'\textrm{ and } \left|1-\frac{g_f(\gamma_t)^{p/q}}{|\gamma_t'|}\right|<\varepsilon,
	\end{align}
	cf. \eqref{eq:genericineq} and Lemma \ref{lem:youngstable}. Once we show that $\ud\bm\pi:=\chi_B|\gamma_t'|\ud t\ud\bm\eta$ satisfies
	\begin{align*}
	\mu|_D\ll e_*\bm\pi\ll \mu|_D,
	\end{align*}
	it follows from \eqref{eq:genericineq} and \eqref{eq:conditions} that $\bm\pi':=\bm\pi/\bm\pi(C(I;X)\times I)\in \mathcal P(C(I;X)\times I)$ satisfies
	\begin{align*}
	(1-\varepsilon)g_f(\gamma_t)|\gamma_t'|\le (f\circ\gamma)_t'\le g_f(\gamma_t)\textrm{ and } \frac{g(\gamma_t)^{p/q}}{1+\varepsilon}\le |\gamma_t'|\le \frac{g(\gamma_t)^{p/q}}{1-\varepsilon}
	\end{align*}
	for $\bm\pi'$-almost every $(\gamma,t)$, which readily implies the inequalities in the theorem.
	
	\bigskip\noindent To prove $e_*\bm\pi\ll \mu|_D$ observe that \eqref{eq:conditions} implies $\chi_B|\gamma_t'|\ud t\ud\bm\eta\le (1+\varepsilon)g(\gamma_t)^{p/q}\ud t\ud \bm\eta$ and thus
	\begin{align*}
	\int\int_0^1\chi_B(\gamma,t)\chi_E(\gamma_t)|\gamma_t'|\ud t\ud\bm\eta\le (1+\varepsilon)\int_0^1\int_X\chi_Eg_f^{p/q}e_{t*}(\ud\bm\eta)\ud t\le C\int_Eg_f^{p/q}\ud\mu
	\end{align*}
	for any Borel set $E\subset X$.

	\bigskip\noindent It remains to prove that $\mu|_D\ll e_*\bm\pi$. Let $E\subset D$ be a Borel set with $\mu(E)>0$. Then $e_{0*}\bm\eta(E)=\bm\eta(\{ \gamma:\ \gamma_0\in E \})>0$. Since 
	\begin{align*}
	0\le \frac 1p\int_0^t\chi_E(\gamma_s)A(\gamma,s)\ud s-D_{E,t}(\gamma)&\le \frac 1pG_t(\gamma)+\frac 1q\widetilde E_t(\gamma)-D_t(\gamma)\stackrel{t\to 0}{\longrightarrow} 0,\\
	D_{E,t}&\stackrel{t\to 0}{\longrightarrow}\chi_Eg_f^p\circ e_0
	\end{align*}
	in $L^1(\bm\eta)$, cf. Lemmas \ref{lem:conv} and \ref{lem:Econv} respectively, there exist a sequence $T_n\to 0$ such that for $\bm\eta$-almost every $\gamma\in e_0\inv(E)$ the functions 
	\[
	h_\gamma(s):=\chi_E(\gamma_s)(f\circ\gamma)_s',\quad g_\gamma(s):=\chi_E(\gamma_s)A(\gamma,s)
	\]
	satisfy the hypotheses of Lemma \ref{lem:posmeas}. It follows that for $\bm\eta$-almost every $\gamma\in e_0\inv(E)$ the sets
	\begin{align*}
	I^n_\gamma:=\{s\in [0,T_n]:\ (1-\delta_0)g_\gamma(s)<h_\gamma(s) \}=\{s\in [0,t_n]: \gamma_s\in E,\ H(\gamma,s)\le 0\}
	\end{align*}
	have positive measure for all $n$. Notice that for $\bm\eta$-almost every $\gamma$, if $s\in I^n_\gamma$ then $\gamma_s\in E$ and $|\gamma_s'|>0,g_f(\gamma_s)>0$ (since $0<(f\circ\gamma)_s'\le g_f(\gamma_t)|\gamma_s'|$). Consequently 
	\[
	\int_0^1\chi_B(\gamma,s)\chi_{E}(\gamma_s)|\gamma_s'|\ud s\ge \int_{I^n_\gamma}|\gamma_s'|\ud s>0
	\]
	for $\bm\eta$-almost every $\gamma\in e_0\inv(E)$ which in turn implies that $e_*\bm\pi(E)>0$. Since $E\subset D$ is an arbitrary Borel set with positive $\mu$-measure, this completes the proof.
\end{proof}


\section{Charts and differentials}
\subsection{Notational remarks}
In what follows, define for any set $U \subset X$ the set of curves which spend positive \emph{length} in $U$:
\[
\Gamma_U^+=\{\gamma\in AC(I;X): \int_\gamma\chi_U \ud s>0 \}.
\]
Having positive length in $U$ is more restrictive than assuming that $\gamma\inv(U)$ has positive measure. We will also discuss $p$-weak differentials and co-vector fields of the form $\ud f: U \to (\R^N)^*$ or $\bm\xi: U \to (\R^N)^*$ for measurable subsets $U\subset X$. The value of such a map at $x\in U$ is denoted $\ud_x f, \bm\xi_x$, respectively.
\subsection{Canonical minimal gradients}

Let $p\ge 1$ and $N\ge 0$ be given. For the next three lemmas we fix $\varphi\in N^{1,p}_{loc}(X;\R^N)\simeq N^{1,p}_{loc}(X)^N$, 
with the convention $N^{1,p}_{loc}(X;\R^N)= N^{1,p}_{loc}(X)^N=\{0\}$ when $N=0$. Our aim is to construct a ``canonical'' representative of the minimal weak upper gradients $|D(\xi \circ \varphi)|_p$ of the functions $\xi\circ\varphi$. We will use a plan to represent it.

\begin{lemma}\label{lem:canonicalplan} 
There exists a $q$-plan $\bm\eta$ and a Borel set $D$ with $\mu|_D \ll \bm\eta^\#$ such that
\begin{equation}\label{eq:planeq}
\Phi_\xi(x):=\chi_D(x)\left\|\frac{\xi((\varphi\circ\gamma)'_t)}{|\gamma'_t|}\right\|_{L^\infty(\bm \pi_x)}
\end{equation}
is a representative of $|D(\xi \circ \varphi)|_p$ for every $\xi \in (\R^N)^*$. Here  $\{\bm\pi_x\}$ is the disintegration of $\ud\bm\pi:=|\gamma_t'|\ud \bm\eta\ud t$ with respect to the evaluation map $e$.

\end{lemma}

\begin{proof}
Let $\{\xi_0,\xi_1,\ldots\}\subset (\R^N)^*$ be a countable dense set and, for each $n\in \N$, choose Borel representatives $\rho_n$ of $|D(\xi_n \circ \varphi)|_p$ and denote $D_n:=\{ \rho_n>0\}$. By Theorem \ref{thm:curvewise-p=1} and the Borel regularity of $\mu$, for each $n\in \N$ there exists a $q$-plan $\bm\eta_n$ and a Borel set $B_n\subset D_n$ with $\mu(D_n\setminus B_n)=0$ such that the disintegration $\{\bm\pi_x^n\}$ of $\ud\bm\pi^n:=|\gamma_t'|\ud\bm\eta_n\ud t$ satisfies
\begin{align*}
\left\|\frac{\xi_n((\varphi\circ\gamma)'_t)}{|\gamma'_t|}\right\|_{L^\infty(\bm \pi_x^n)}=\rho_n(x)
\end{align*}
for every $x\in B_\xi$.

Define $D:=\bigcup_{n\in\N}B_n$ and $\bm \eta = \sum_{n} 2^{-n}a_n^{-1} \bm \eta_{n}$, where $a_n=1+ \bm \eta_n(C(I;X)) + ||\frac{d\bm \eta_{n}^\#}{d\mu}||_{L^q} + \bm\pi^n(AC(I;X)\times I)$. Then $\mu|_D\ll\bm\eta^\#$. Define $\Phi_\xi(x)$ as in Equation \eqref{eq:planeq}. By Lemma \ref{lem:uppergradient} we have that $\rho_n=\Phi_{\xi_n}$ $\mu$-a.e. on $X$ and thus the claim holds for every $\xi_n\in A$. 

We prove the claim in the statement for arbitrary $\xi\in(\R^N)^*$. Let $(\xi_{n_l})_l\subset A$ be a sequence with $|\xi_{n_l}-\xi|<2^{-l}$ and denote by $\varphi_1,\ldots,\varphi_N\in N^{1,p}(X)$ the component functions of $\varphi$. Since
\begin{align*}
||D(\xi_{n_l}\circ\varphi)|_p-|D(\xi\circ\varphi)|_p|\le |D((\xi_{n_l}-\xi)\circ\varphi)|_p\le |\xi_{n_l}-\xi|\sum_k^N|D\varphi_k|_p
\end{align*}
$\mu$-a.e., we have that $\displaystyle |D(\xi\circ\varphi)|_p=\lim_{l\to\infty}\Phi_{\xi_{n_l}}$  $\mu$-a.e. on $X$. In particular, $|D(\xi\circ\varphi)|_p=0$ $\mu$-a.e. on $X\setminus D$. On the other hand, for $p$-a.e. curve $\gamma$, we have that $$|\xi_{n_l}((\varphi\circ\gamma)_t')-\xi((\varphi\circ\gamma)_t')|\le |\xi_{n_l}-\xi|\sum_k^N|D\varphi_k|_p(\gamma_t)|\gamma_t'| \quad \textrm{a.e. }t.$$ 
Since $\bm\eta$ is a $q$-plan with $\mu|_D\ll\bm\eta^\#$ this implies that 
\begin{align*}
\limsup_{l\to\infty}\left|\frac{\xi_{n_l}((\varphi\circ\gamma)'_t)}{|\gamma'_t|}-\frac{\xi((\varphi\circ\gamma)'_t)}{|\gamma'_t|}\right|\le \limsup_{l\to\infty}|\xi_{n_l}-\xi|\sum_k^N|D\varphi_k|_p(x)=0\quad\bm\pi_x-\textrm{a.e. }(\gamma,t)
\end{align*}
for $\mu$-a.e. $x\in D$. Thus $\Phi_\xi(x)=\lim_{l\to\infty}\Phi_{\xi_{n_l}}(x)$ for $\mu$-a.e. $x\in D$. Since $\Phi_\xi=0=|D(\xi\circ \varphi)|_p$ $\mu$-a.e. on $X\setminus D$, the proof is completed.
\end{proof}

In the next two lemmas we collect the properties of the Borel function constructed above.
\begin{lemma}\label{lem:representative} The map $\Phi:(\R^{N})^* \times X \to \R$ given by \eqref{eq:planeq} is Borel and satisfies the following.
	\begin{enumerate}
		\item[(1)] For every $\xi\in (\R^N)^*$, $\Phi_\xi:=\Phi(\xi,\cdot)$ is a representative of $|D(\xi\circ\varphi)|_p$, and 
		\item[(2)] for every $x\in X$, $\Phi^x:=\Phi(\cdot,x)$ is a seminorm in $(\R^N)^*$.
	\end{enumerate}
	Moreover, there exists a path family $\Gamma_B$ with $\Mod_p(\Gamma_B)=0$ and for each $\gamma\in AC(I;X) \setminus \Gamma_B$ a null-set $E_\gamma \subset I$ so that, for {\bf every } $\xi\in(\R^N)^*$, we have 
	\begin{itemize}
		\item[(3)] $\Phi_\xi$ is an upper gradient of $\xi\circ\varphi$ along $\gamma$, and 
		\item[(4)]\label{eq:derivativebound} $|(\xi \circ \varphi \circ \gamma)_t'| \leq \Phi_\xi(\gamma_t) |\gamma'_t|$ for $t\notin E_\gamma$.
	\end{itemize}
\end{lemma}
\begin{proof}[Proof of Lemma \ref{lem:representative}]

Borel measurability follows from Lemmas \ref{lem:measurable-integral} and Corollary \ref{cor:borel}, and property (1) follows from Lemma \ref{lem:canonicalplan}, while (2) follows from Equation \eqref{eq:planeq}.

Fix a countable dense set $A \subset (\R^N)^*$ and one $\xi \in A$. We have that $\Phi(\xi,x)$ is a weak upper gradient for $\xi \circ\varphi$, so there is  family of curves $\Gamma_\xi$ so that $\xi \circ \varphi$ is absolutely continuous with upper gradient $|D(\xi \circ \varphi)|_p$ on each $\gamma \in \Gamma_i$, and so that $\Mod_p(\Gamma \setminus \Gamma_i) =0$. Let $\Gamma' = \cap_{\xi \in A} \Gamma_\xi$, whose complement $\Gamma_B =AC(I;X) \setminus \Gamma'$ has null $p$-modulus. 

Since $\xi\circ\varphi$ has as upper gradient $\Phi_\xi(x)$ on $\gamma$ for each $\xi \in A$, then by considering a sequence $\xi_l\to \xi$ for $\xi \not\in A$ we obtain the same conclusion. 

Finally, fixing an absolutely continuous curve $\gamma\not\in \Gamma_B$ there is a full measure set $F_\gamma^1$, where the components of $\varphi\circ\gamma_t$ are differentiable at $t\in F_\gamma^1$. Both sides of (4) are continuous and defined in $\xi$ on the set $F_\gamma^1$. Since $\Phi_\xi(x)$ is an upper gradient for $\xi\circ\varphi$ along $\gamma$, there is a full measure subset $F_\gamma\subset F_\gamma^1$, where the inequality holds for $\xi \in A$. Continuity then extends it for all $\xi \in (\R^N)^*$ and $t\in F_\gamma$ and the claim follows by setting $E_\gamma=I \setminus F_\gamma$.

\end{proof}

Next, we collect some basic properties of the canonical minimal gradient. Let $\Phi$ be the map given by \eqref{eq:planeq}.

\begin{lemma}\label{lem:properties}
	Set $\displaystyle I(\varphi)(x):=\inf_{\|\xi\|_*=1}\Phi^x(\xi)$ for $\mu$-a.e. $x\in X$. Then
\begin{enumerate}
	\item $\displaystyle I(\varphi)=\essinf_{\|\xi\|_*=1}\,|D(\xi\circ\varphi)|_p$ $\mu$-a.e. in $ X$;
	\item If $U\subset X$ and $\bm\xi:U\to (\R^N)^*$ are Borel, then $\Phi^x(\bm\xi_x)=0$ $\mu$-a.e. $x\in U$ if and only if $\bm\xi_{\gamma_t}((\varphi\circ\gamma)_t')=0$ a.e. $t\in \gamma\inv(U)$, for $p$-a.e. absolutely continuous $\gamma$ in $X$.
	\item If $\varphi$ is $p$-independent on $U$ and $f\in \Ne pX$, then the $p$-weak differential $\ud f$ with respect to $(U,\varphi)$, if it exists, must be unique.
\end{enumerate}
\end{lemma}

\begin{proof}[Proof of Lemma \ref{lem:properties}]
First, we show (1). For any $\xi$ in the unit sphere of $(\R^N)^*$, we have $\Phi_\xi(x)=|D(\xi \circ \varphi)|_p$ almost everywhere by Lemma \ref{lem:canonicalplan}. Taking an infimum on the left then gives
$$\inf_{\|\zeta\|_*=1} \Phi_\zeta(x)\leq |D(\xi \circ \varphi)|_p,$$
i.e. $\inf_{\|\zeta\|_*=1} \Phi_\zeta(x) \leq \essinf_{\|\xi\|_*=1}\, |D(\xi \circ \varphi)|_p$ almost everywhere by the definition of an essential infimum, see Definition \ref{def:essup}.  

On the other hand, if $\xi_n$, for $n\in\N$, is a countably dense collection in the unit sphere of $(\R^N)^*$, then we have $\Phi_{\xi_n}(x)=|D(\xi_n \circ \varphi)|_p\geq \essinf_{\|\xi\|_*=1} \,|D(\xi \circ \varphi)|_p$ almost everywhere. By intersecting the sets where this holds for different $\xi_n$ and since the collection is countable, we have that these hold simultaneously on a full-measure set.
 Specifically, $\inf_{n\in\N} \Phi_{\xi_n}(x) \geq \essinf_{\|\xi\|_*=1}\, |D(\xi \circ \varphi)|_p$. By Lemma \ref{lem:representative}, we have that $\xi\to \Phi_\xi(x)$ is Lipschitz. Thus, almost everywhere,
$$\inf_{\|\xi\|_*=1} \Phi(\xi,x) = \inf_{n\in\N}\Phi(\xi_n,x)\geq \essinf_{\|\xi\|_*=1} |D(\xi \circ \varphi)|_p,$$ which gives the claim.

Next fix $\bm\xi:U \to (\R^N)^*$ as in (2). Assume first that $\Phi^x(\bm\xi_x)=0$ for $\mu$-a.e. $x\in U$. Set $C = \{x : \Phi^{x}(\bm\xi_{\gamma_x})\neq 0\}$ with $\mu(C) = 0$. Since $\mu(C)=0$, we have $\Mod_p(\Gamma^+_C)=0$. Let $\Gamma_B$ be the family of curves from Lemma \ref{lem:representative}. We will show the claim for $\gamma\in AC(I;X)\setminus (\Gamma_B\cup \Gamma^+_C)$. By Lemma \ref{lem:representative}(4), we obtain a null set $E_\gamma$ so that for any $\xi\in (\R^N)^*$ we have
$|(\xi \circ \varphi \circ \gamma)_t'| \leq \Phi_\xi(\gamma_t) |\gamma'_t|$ and $t\notin E_\gamma$. Let $F_\gamma$ be the set of $t\not\in E_\gamma$ so that $|\gamma'_t|>0$ and $\Phi^{\gamma_t}(\bm\xi_{\gamma_t}) \neq 0$.  Since $0=\int_{\gamma} 1_C \ud s \geq \int_{F_ \gamma} |\gamma_t'| \ud t$, we have that the measure of $F_\gamma$ is null. 
Now, if $t\not\in E_\gamma \cup F_\gamma$, then either $|\gamma'_t|=0$ (and the condition is vacuously satisfied), or the claim follows from $\Phi^{\gamma_t}(\bm\xi_{\gamma_t})=0$.

On the other hand, suppose that $\bm\xi_{\gamma_t}((\varphi\circ\gamma)_t')=0$ for a.e. $t\in \gamma\inv(U)$ and $p$-a.e. absolutely continuous curve $\gamma$. Let $\bm\eta$ be the $q$-plan from Lemma \ref{lem:canonicalplan} and $\{\bm\pi_x\}$ the disintegration given there. The equality $\bm\xi_{\gamma_t}((\varphi\circ\gamma)_t')=0$ holds then for $\bm\eta$-a.e. curve and a.e. $t\in\gamma\inv(U)$, since $\bm \eta$ is a $q$-plan (recall Remark \ref{rmk:modnull}). Then for $\mu$-a.e. $x$ we have $\Phi_{\bm\xi}(x)=0$ or we have $\Phi_{\bm\xi_x}(x)=\left\|\frac{\bm \xi_x((\varphi\circ\gamma)'_t)}{|\gamma'_t|}\right\|_{L^\infty(\bm \pi_x)}$. 
In the latter case, since $\bm\eta$ is a $q$-plan, we have for $\mu$-a.e. such $x$ and $\bm\pi_x$-a.e. $(\gamma,t) \in {\rm Diff}(f)\cap e^{-1}(x)$  that $\bm\xi_{x}((\varphi\circ\gamma)_t')=0$. Thus, the claim follows together with the properties of disintegrations and Corollary \ref{cor:borel}, since the essential supremum then vanishes. 

The final claim about uniqueness follows since, if $\ud_i f$ were two $p$-weak differentials for $i=1,2$, then we could define $\bm\xi_x=\frac{\ud_1 f-\ud_2 f}{\|\ud_1 f-\ud_2 \|_{x,*}}$ when $\ud_1 f \neq \ud_2 f$ and otherwise $\bm\xi_x = 0$. We then get immediately form the definition and the second part that $\Phi^x(\bm\xi) = 0$, for $\mu$-a.e. $x\in U$. This would contradict independence.
 
\end{proof}

\subsection{Charts} The presentation here should be compared to \cite[Section 4]{che99}, and specifically to the proof of \cite[Theorem 4.38]{che99}, where similar arguments are employed. We first consider 0-dimensional $p$-weak charts. These correspond to regions of the space where no curve spends positive time. 
\begin{prop}\label{prop:zerochart}
Suppose $(U,\varphi)$ is a 0-dimensional $p$-weak chart. Then 
\begin{align}\label{eq:modzero}
\Mod_p(\Gamma_U^+)=0.
\end{align}
Conversely, if $U\subset X$ is Borel and satisfies \eqref{eq:modzero}, then $(U,0)$ is a 0-dimensional $p$-weak chart of $X$. 
\end{prop}
\begin{proof}
Since $(U,\varphi)$ is a 0-dimensional $p$-weak chart, we have that 
\begin{align}\label{eq:0chart}
|Df|_p=0\quad \mu-\textrm{a.e. in }U
\end{align}
for every $f\in \LIP_b(X)$. Let $\{x_n\}\subset X$ be a countable dense subset, and $f_n:=\max\{1-d(x_n,\cdot),0\}$. By \cite[Thm 1.1.2]{AGS08} (see also its proof) and \eqref{eq:0chart} we have that $$|\gamma_t'|=\sup_n|(f_n\circ\gamma)_t'|\le \sup_n|Df_n|_p(\gamma_t)|\gamma_t'|=0 \quad a.e\ t\in \gamma\inv(U)$$
for $p$-a.e. $\gamma\in AC(I;X)$. It follows that $\int_\gamma\chi_U\ud s=0$ for $p$-a.e. $\gamma\in AC(I;X)$, proving \eqref{eq:modzero}.

In the converse direction, \eqref{eq:modzero} implies, for any $f\in \LIP_b(X)$, that
\[ \int_0^1\chi_U(\gamma_t)|(f\circ\gamma)_t'|\ud t\le \LIP(f)\int_0^1\chi_U(\gamma_t)|\gamma_t'|\ud t=0\]
for $p$-a.e $\gamma\in AC(I;X)$. Thus $|(f\circ\gamma)_t'|=0$ for $p$-a.e. $\gamma\in AC(I;X)$ and a.e. $t\in\gamma\inv(U)$. Then, by Theorem \ref{thm:curvewise-p=1}, together with measurability considerations from Corollary \ref{cor:borel}, this gives $|Df|_p=0$ $\mu$-a.e. on $U$ for every $f\in \LIP_b(X)$, showing that $(U,0)$ is a 0-dimensional $p$-weak chart. 
\end{proof}

For the remainder of this subsection we assume that $N\ge 1$ and that $(U,\varphi)$ is an $N$-dimensional chart of $X$. Denote by  $\Phi$ the canonical minimal gradient of $\varphi$ (see Lemma \ref{lem:canonicalplan}).

\begin{lemma}\label{lem:difflip}
	The function $\xi \mapsto \Phi^x(\xi)$ is a norm on $(\R^N)^*$ for $\mu$-a.e. $x\in U$. Moreover, for every $f\in \LIP(X)$ there exists a $p$-weak differential $\ud f$. That is, a Borel measurable map $\ud f:U\to (\R^N)^*$ satisfying
	\begin{align*}
	(f\circ\gamma)_t'=\ud_{\gamma_t}f((\varphi\circ\gamma)'_t)\quad a.e. \ t\in \gamma\inv(U),
	\end{align*}
	for $p$-a.e. absolutely continuous curve $\gamma$ in $X$. The map $\ud f$ is uniquely determined a.e. in $U$ and satisfies $|Df|_p(x)=\Phi^x(\ud f)$ $\mu$-a.e. in $U$.
\end{lemma}
\begin{remark}\label{rmk:curv-domain} The equation in the statement is an equivalent formulation of the definition of the $p$-weak differential in Definition \ref{def:p-weakdiff}. Indeed, the latter follows by integration of the first, and conversely, the first follows by Lebesgue differentiation. Further, it would be enough to consider only p-a.e. curve $\gamma \in \Gamma_U^+$. Indeed, if a curve $\gamma$ does not spend positive length in the set $U$, then $|\gamma_t'|=0$ for a.e. $t\in \gamma\inv(U)$ and both sides of the equation vanish.
\end{remark}
\begin{proof} First, consider $f\in \LIP_b(X)$.
	Since $\Phi^x$ is a norm if and only if $I(\varphi)(x)>0$, Lemma \ref{lem:properties}(1) and \eqref{eq:weakcha} imply that $\Phi^x$ is a norm for $\mu$-a.e. $x\in U$. 
	
	Next, let $f\in \LIP_b(X)$ and consider the map $\psi=(\varphi,f):X\to \R^{N+1}$. Let $\Psi$ be the canonical minimal gradient of $\psi$. Given $\xi\in (\R^N)^*$ and $a\in \R$, we use the notation $$(\xi,a)\in(\R^{N+1})^*,\quad v=(v',v_{N+1})\mapsto \xi(v')+av_{N+1}.$$
	For $\mu$-a.e. $x\in U$, we have that $\Psi^x(\xi,0)=\Phi^x(\xi)$ and $\Psi^x(0,a)=|a||Df|_p(x)$ for every $\xi\in (\R^N)^*$, $a\in \R$ (cf. Lemma \ref{lem:representative}(3) and (4)). Since $\varphi$ is a chart, we have that $I(\psi)=0$ almost everywhere. Thus, given that $I(\varphi)>0$,  $\ker\Psi^x$ is a 1-dimensional subspace of $(\R^{N+1})^*$. 

	Thus for $\mu$-a.e. $x\in U$ there exists a unique $\xi:=\ud_xf\in (\R^N)^*$ such that $\Psi^x(\ud_xf,-1)=0$, and the map $x\mapsto \ud_xf$ is Borel, see e.g. \cite[Lemma 6.7.1]{bogachev07}. By Lemma \ref{lem:properties}(2), $\ud f:U\to(\R^N)^*$ satisfies
	\[
	0=(\ud_{\gamma_t}f,-1)((\psi\circ\gamma)_t')= \ud_{\gamma_t}f((\varphi\circ\gamma)_t')-(f\circ\gamma)_t' \quad a.e.\ t\in \gamma\inv(U),
	\]
	for $p$-a.e. $\gamma$. Moreover, we have
	\[
	\left| |Df|_p(x)-\Phi^x(\ud_xf)\right|\le \left|\Psi^x(0,-1)-\Psi^x(\ud_xf,0)\right|\le \Psi^x(\ud_xf,-1)=0
	\]
	for $\mu$-a.e. $x\in U$, completing the proof in the case $f\in \LIP_b(X)$.
	
	The case of $f\in \LIP(X)$ follows through localization. Indeed, let $x_0\in X$ be arbitrary, and consider the functions $\eta_n(x) \defeq \min\{\max\{n-d(x_0,d),0\},1\}$ for $n\in \N$. Then, define $f_n = \eta_n f$ so that $f_n|_{B(x_0,n-1)} = f|_{B(x_0,n-1)}$. For each $f_n$ we can define a differential $\ud f_n$, and $\ud f_n|_{B(x_0,\min(m,n)-1)} = \ud f_m |_{B(x_0,\min(m,n)-1)}$ (a.e.) for each $n,m\in \N$. Thus, we can define $\ud f(x) = \ud f_n(x)$ for $x\in B(x_0,n-1)$ with only an ambiguity on a null set. It is easy to check that $\ud f$ is a differential.
\end{proof}

\subsection{Differential and pointwise norm}

Denote $|\cdot|_x:=\Phi^x$ and define $$\Gamma_p(T^*U)=\{\bm\xi:U\to (\R^N)^*\ \textrm{Borel}:\ \|\bm\xi\|_{\Gamma_p(T^*U)}<\infty\},\quad \|\bm\xi\|_{\Gamma_p(T^*U)}:=\left(\int_U|{\bm\xi}|_x^p\ud\mu\right)^{1/p},$$
(with the usual identification of elements that agree $\mu$-a.e.). Then $(\Gamma_p(T^*U),\|\cdot\|_{\Gamma_p(T^*U)})$ is a normed space. Observe that, if $V_j:=U\cap \{I(\varphi)\ge 1/j\}$, the sets $U_j:=V_j\setminus\bigcup_{i<j}V_i$ partition $U$ up to a null-set and we have an isometric identification
\begin{equation}\label{eq:isomdecomp}
\Gamma_p(T^*U)\simeq \bigoplus_{l^p}\Gamma_p(T^*U_j),\quad \textrm{where}\quad \Gamma_p(T^*U_j)\simeq L^p(U_j;(\R^N)^*).
\end{equation}
Thus $(\Gamma_p(T^*U),\|\cdot\|_{\Gamma_p(T^*U)})$ is a Banach space.  Recall, that an $\ell_p$-direct sum of Banach spaces $B_i$ with norms $\|\cdot\|_{B_i}$ with countable index set $I$ is defined by
$$\bigoplus_{\ell^p} B_i \defeq \{(v_i)_{i\in I}: \|(v_i)_{i\in I}\|=\left(\|v_i\|_{B_i}^p \right)^{1/p}, v_i \in B_i\}.$$

\begin{lemma}\label{lem:diff}
	Suppose $(f_n)\subset \LIP_{b}(X)$ is a sequence such that $f_n\to f$ in $L^p(X)$ and $\ud f_n\to \bm\xi$ in $\Gamma_p(T^*U)$ for some $f\in N^{1,p}(X)$ and $\bm\xi\in \Gamma_p(T^*U)$. Then $\bm\xi$ is the (uniquely defined) differential of $f$ in $U$, and
	\begin{align*}
\lim_{n\to\infty}\int_U|D(f_n-f)|_p^p\ud\mu=0.
	\end{align*}
	In particular, $\Phi(\bm\xi,\cdot)=|Df|_p$ $\mu$-a.e. in $U$.
\end{lemma}
\begin{proof}
By Lemma \ref{lem:difflip} and Fuglede's Theorem \cite[Theorem 3(f)]{fuglede1957} (applied to the sequence of functions $h_n=\chi_U(\gamma_t)|\ud_{\gamma_t}f_n-\bm\xi_{\gamma_t}|_{\gamma_t}$ and $f_n$) we can pass to a subsequence so that
\begin{align}\label{eq:curvewiseconv}
\lim_{n\to\infty}&\int_0^1\chi_U(\gamma_t)|(f_n\circ\gamma)_t'-\bm\xi_{\gamma_t}((\varphi\circ\gamma)'_t)|\ud t\le\lim_{n\to\infty} \int_0^1\chi_U(\gamma_t)|\ud_{\gamma_t}f_n-\bm\xi_{\gamma_t}|_{\gamma_t}|\gamma_t'|\ud t=0,\nonumber\\
\lim_{n\to\infty}&\int_0^1|f_n(\gamma_t)-f(\gamma_t)||\gamma_t'|\ud t=0
\end{align}
for $p$-a.e. $\gamma\in AC(I;X)$. Fix a curve $\gamma$ where \eqref{eq:curvewiseconv} holds and $f_n\circ\gamma$, $f\circ\gamma$ are absolutely continuous. We may assume that $\gamma$ is constant speed parametrized. By \eqref{eq:curvewiseconv} $f_n\circ\gamma\to f\circ\gamma$ in $L^1([0,1])$ and $(f_n\circ\gamma)'\to g$ in $L^1(\gamma\inv(U))$, where $g(t):= \chi_U(\gamma_t)\bm\xi_{\gamma_t}((\varphi\circ\gamma)_t')$. It follows that
\[
(f\circ\gamma)_t'=\bm\xi_{\gamma_t}((\varphi\circ\gamma)_t')\quad \textrm{ a.e. }t\in \gamma\inv(U).
\]
This shows that $\bm\xi$ is the differential of $f$, and uniqueness follows from Lemma \ref{lem:properties}(3). The identity  $((f-f_n)\circ\gamma)_t'=(\bm\xi_{\gamma_t}-\ud f_n)((\varphi\circ\gamma)_t')$ a.e. $t\in \gamma\inv(U)$ for  $p$-a.e. $\gamma\in AC(I;X)$, together with Lemma \ref{lem:uppergradient}(3), 
implies that $\Phi^x(\bm\xi-\ud f_n)\le |D(f-f_n)|_p$ for $\mu$-a.e. $x\in U$. By the convergence $\ud (f_m-f_n)\to\bm\xi-\ud f_n$ (as $m\to\infty$) we have that $|D(f_m-f_n)|_p\to_{m\to\infty} \Phi^x(\bm\xi-\ud f_n)$ in $L^p(U)$, and thus $|D(f-f_n)|_p\le \Phi^x(\bm\xi-\ud f_n)$ $\mu$-a.e. in $U$. Thus $|D(f-f_n)|_p=\Phi^x(\bm\xi-\ud f_n)$ converges to zero in $L^p(U)$. The equality $\Phi_{\bm\xi}=|Df|_p$ follows, completing the proof.
\end{proof}

We say that a sequence $(\bm\xi_n)_n\subset\Gamma_p(T^*U)$ is equi-integrable if the sequence $\{|\bm\xi_n|_x\}_n\subset L^p(U)$ is equi-integrable. 
Recall, that a collection of integrable functions $\mathcal{F}$ is called equi-integrable, if there is a $M$ so that $\int_X |f|^p~d\mu \leq M$ for every $f\in \mathcal{F}$ and if for every $\epsilon>0$, there is an $\delta>0$ and a positive measure subset $\Omega_\epsilon$, so that for any measurable set $E$ with $\mu(E)\leq \delta$, we have $\int_{\Omega_\epsilon^c\cup E} |f|^p \ud \mu \leq \epsilon$ for each $f\in \mathcal{F}$. By the Dunford-Pettis Theorem a set of $L^1$ functions is equi-integrable if and only if it is sequentially compact, see for example \cite[Theorem IV.8.9]{dunsch}.

\begin{remark}\label{rmk:equi-int}
	It follows from \eqref{eq:isomdecomp} that, if $(\bm\xi_n)_n\subset \Gamma_p(T^*U)$ is equi-integrable, then there exists $\bm\xi\in\Gamma_p(T^*U)$ such that $\bm\xi_n\rightharpoonup \bm\xi$ weakly in $\Gamma_p(T^*U)$ up to a subsequence and, by Mazur's lemma, that a convex combination of $\bm\xi_n$'s converges to $\bm\xi$ in $\Gamma_p(T^*U)$. Indeed, the $p>1$ case is direct and the $p=1$ case uses the Dunford-Pettis argument above.
\end{remark}

Next, we show that any Sobolev function $f\in N^{1,p}$ has a uniquely defined differential with respect to a chart. Note, however, that here we still postulate the existence of charts.

\begin{proof}[Proof of Theorem \ref{thm:differentiability}]
The measurable norm $|\cdot|_x$ is given by Lemma \ref{lem:difflip}. Let $f\in \Ne pX$. Lemma \ref{lem:properties}(3) implies that $\ud f$, if it exists, is a.e. uniquely determined on $U$. Let $(f_n)\subset \LIP_b(X)$ be such that $f_n\to f$ and $|Df_n|_p\to |Df|_p$ in $L^p(\mu)$ as $n\to\infty$, which exists by \cite[Theorem 1.1]{seb2020}. By Lemma \ref{lem:difflip}, $(\ud f_n)_n\subset \Gamma_p(T^*U)$ is equi-integrable. It follows that there exists $\bm\xi\in\Gamma_p(T^*U)$ such that $\ud f_n\rightharpoonup \bm\xi$ weakly in $L^p(T^*U)$, cf. Remark \ref{rmk:equi-int}. By Mazur's lemma, a sequence $(g_n)\subset \LIP_b(X)$ of convex combinations of the $f_n$'s converges to $f$ in $L^p(\mu)$ and $\ud g_n\to\bm\xi$ in $\Gamma_p(T^*U)$. By Lemma \ref{lem:diff}, $\bm\xi=:\ud f$ is the differential of $f$. The linearity of $f\mapsto \ud f$ follows from the uniqueness of differentials, Lemma \ref{lem:properties}(3).
\end{proof}
The proof above also yields the following corollary. Note that, while the claim initially holds only after passing to a subsequence, since the limit is unique, the convergence holds along the full sequence.
\begin{cor}\label{cor:localnormdensity}
	Let $(U,\varphi)$ be a $p$-weak chart of $X$. Suppose that $f\in \Ne pX$ and $(f_n)\subset \LIP_b(X)$ converges to $f$ in energy, that is, $f_n\to_{L^p} f$ and $|Df_n|_p \to_{L^p} |Df|_p$. Then we have that $\ud f_n \rightharpoonup \ud f$ weakly in $\Gamma_p(T^*U)$. 
\end{cor}

Using Lemma \ref{lem:properties} we prove that the differential satisfies natural rules of calculation.

\begin{prop}\label{prop:calcrules}
	Let $(U,\varphi)$ be an $N$-dimensional $p$-weak chart of $X$, $f,g\in \Ne pX$, and $F:X\to Y$ a Lipschitz map into a metric measure space $(Y,d,\nu)$ with $F_\ast\mu\le C\nu$ for some $C>0$.
	\begin{enumerate}
	\item If $(V,\psi)$ is a $p$-weak chart with $\varphi|_{U\cap V}=\psi|_{U\cap V}$ then the $p$-weak differentials of $f$ with respect to both charts agree $\mu$-a.e. on $U\cap V$.
	\item If $f,g\in L^\infty(X)$, then $\ud(fg)=f\ud g+g\ud f$ $\mu$-a.e. on $U$.
	\item Let $(V,\psi)$ be an $M$-dimensional $p$-weak chart of $Y$ with $\nu(U\cap F\inv(V))>0$. For $\mu$-a.e. $U\cap F\inv(V)$ there exists a unique linear map $D_xF:\R^N\to \R^M$ satisfying the following: if $h\in \Ne pY$ and $E$ is the set of $y\in V$ where the differential $\ud_yh$ does not exist, then $\mu(U\cap F\inv(E))=0$ and
	$$
	\ud_x(h\circ F)=\ud_{F(x)}h\circ D_{x}F	\quad\mu-\textrm{a.e. }x\in U\cap F\inv(V\setminus E).
	$$
	\end{enumerate}
\end{prop}
\begin{proof}
	Claim (1) follows from Lemma \ref{lem:properties}(2) and the fact that $(\varphi\circ\gamma)_t'=(\psi\circ\gamma)_t'$ a.e. $t\in \gamma\inv(U\cap V)$, for $p$-a.e. $\gamma \in AC(I;X)$. 
	
	To prove (2) note that, since we have $(fg\circ\gamma)_t'=g(\gamma_t)(f\circ\gamma)_t'+f(\gamma_t)(g\circ\gamma)_t'$ a.e $t$ for $p$-a.e. curve $\gamma \in AC(I;X)$, it follows from \eqref{eq:differential} that $$\ud_{\gamma_t}(fg)((\varphi\circ\gamma)_t')=g(\gamma_t)\ud f_{\gamma_t}((\varphi\circ\gamma)_t')+f(\gamma_t)\ud g_{\gamma_t}((\varphi\circ\gamma)_t')\quad a.e.\ t\in \gamma\inv(U)$$ for $p$-a.e. $\gamma\in AC(I;X)$. By Lemma \ref{lem:properties}(2) and (3) the claimed equality holds.
	
	Finally, for (3), let $G=(G_1,\ldots,G_M)=\psi\circ F\in \LIP(X;\R^M)$ and define the expression $D_xF:=(\ud_x G_1,\ldots,\ud_x G_M):\R^N\to \R^M$ for $\mu$-a.e. $x\in U\cap F\inv(V)$. We have that
	$$
	(\psi\circ F\circ\gamma)_t'=D_{\gamma_t}F((\varphi\circ\gamma)_t')\quad a.e.\ t\in \gamma\inv(U)
	$$
	for $p$-a.e. $\gamma \in AC(I;X)$. Note that if $h$ and $E$ are as in the claim, then $\mu(U\cap F\inv(E))\le C\nu(E)=0$. To show the claimed identity, let $\Gamma_0\subset C(I;Y)$ be a path family with $\Mod_p\Gamma_0=0$ such that
	$$
	(h\circ\alpha)_t'=\ud_{\alpha_t}h((\psi\circ\alpha)_t') \quad a.e.\ t\in \alpha\inv(V)
	$$
	for every absolutely continuous $\alpha\notin \Gamma_0$, and set $\Gamma_1=F\inv\Gamma_0:=\{\gamma\in C(I;X):\ F\circ\gamma\in \Gamma_0\}$. Since $\Mod_{p}\Gamma_1\le C\LIP(F)^p\Mod_p(\Gamma_0)=0$ it follows from the two identities above that 
$$
(h\circ F\circ\gamma)_t'=\ud_{F(\gamma_t)}h((\psi\circ F\circ\gamma)_t')=\ud_{F(\gamma_t)}h(D_{\gamma_t}((\varphi\circ\gamma)_t'))\quad a.e.\ t\in \gamma\inv(U\cap F\inv(V))
$$
for $p$-a.e. $\gamma\in AC(I;X)$. Lemma \ref{lem:properties}(2) and (3) imply the claim.	
\end{proof}

\subsection{Dimension bound}\label{subsec:dimbound}

In this section we give a geometric condition which guarantees that finite dimensional weak $p$-charts exist. This involves a bound on the size of $p$-independent Lipschitz maps. 

As a technical tool we need the notion of a decomposability bundle $V(\nu)$ of a Radon measure $\nu$ on $\R^m$, see \cite{almar16}. We will not fully define this here, as we only need some of its properties. Firstly, let ${\rm Gr}(m)$ be the set of linear subspaces of $\R^m$ equipped with a metric $d(V,V')$ defined as the Hausdorff distance of $V \cap \overline{B(0,1)}$ to $V' \cap \overline{B(0,1)}$. The linear dimension of a subspace $V$ is denoted ${\rm dim}(V)$. 
The decomposability bundle is then a certain Borel measurable map $\R^m \to {\rm Gr}(m)$, which associates to every $x\in \R^m$ a subspace $V(\nu)_x \in {\rm Gr}(m)$. In a sense, this bundle measures the directions in which a Lipschitz function must be differentiable in (at almost every point). We collect the main properties we need for this bundle and briefly cite where the proofs of these claims can be found. 

\begin{thm}\label{thm:decompbundle} Suppose that $\nu$ is a Radon measure on $\R^m$, then there exists a decomposability bundle $V(\nu)$ with the following property.
	
	\begin{enumerate}
		\item If ${\rm dim}(V(\nu)_x) = m$ for $\nu$-a.e. $x\in \R^m$, then $\nu \ll \lambda$.
		\item There is a Lipschitz function $f:\R^m \to \R$ so that for $\nu$-a.e. $x\in \R^m$ we have that the directional derivative of $f$ does not exist in the direction $v$ for any $v\not\in V(\nu)_x$.
		\item If $\nu'\ll\nu$, then $V(\nu')_x=V(\nu)_x$ for $\nu'$-a.e. $x\in\R^m$.
	\end{enumerate}
\end{thm}

\begin{proof} The first follows from \cite[Theorem 1.14]{de2016structure} when combined with \cite[Theorem 1.1(i)]{almar16}. The second claim follows from \cite[Theorem 1.1(ii)]{almar16}. Note that the second claim is vacuous for those points $x\in \R^m$ where the decomposability bundle has dimension $m$. The third claim is \cite[Proposition 2.9(i)]{almar16}.
\end{proof}

The following lemma gives a modulus perspective to the decomposability bundle.

\begin{lemma}\label{lem:planinc} Assume $N\geq 1$, $\varphi: X\to\R^N$ is Lipschitz, $U \subset X$ is a Borel set of bounded measure and $\nu =\varphi_*(\mu|_U)$. Then, for $p$-a.e. curve $\gamma$ and almost every $t\in \gamma\inv(U)$ we have that $(\varphi\circ\gamma)_t'$ exists and $(\varphi \circ \gamma)_t'\in V(\nu)_{\varphi(\gamma_t)}$.
\end{lemma}

\begin{proof} By part (ii) of Theorem \ref{thm:decompbundle}, there is a Lipschitz function $f:\R^N\to \R$, so that for $\nu$-almost every $x \in \R^N$ and any $v \not\in V(\nu)_{x}$ we have that the directional derivative $D_v(f)=\lim_{h\to 0} \frac{f(x+hv)-f(x)}{h}$ does not exist. Let $A \subset \R^N$ be a full $\nu$-measure Borel set so that this claim holds. 
	
	Let $B = \varphi\inv(\R^N \setminus A)\cap U$, which is $\mu$-null. The family $\Gamma_B^+$ has null $p$-modulus. We will show that the claim holds for $p$-a.e. $\gamma\in 
	AC(I;X)\setminus \Gamma_B^+$. The derivatives $(\varphi \circ \gamma)'_t$ and $(f\circ\varphi\circ \gamma)'_t$ exist for almost every $t\in \gamma\inv(U)$. Also, for a.e. $t \in I$ we can either take $|\gamma_t'|=0$ or $\gamma_t\not\in B$ and so $(\varphi\circ\gamma)_t \not\in A$, since $\gamma\not\in \Gamma_B^+$. If $|\gamma_t'|=0$, then $(\varphi\circ \gamma)_t' = 0 \in V(\nu)_{\varphi(\gamma_t)}$. In the other case, when $\gamma_t\not\in B$, the function $f$ does not have a directional derivative for $v\not\in V(\nu)_{(\varphi\circ \gamma)_t}$. The only way for both $(\varphi \circ \gamma)'_t$ and $(f\circ\varphi\circ \gamma)'_t$ to exist then is if $(\varphi \circ \gamma)'_t \in V(\nu)_{\varphi(\gamma_t)}$ which gives the claim.
	
	
	
\end{proof}

The following should be compared to \cite[Lemma 4.37]{che99}.

\begin{prop}\label{prop:chahausd}
	Suppose $\varphi\in\LIP(X;\R^N)$ is $p$-independent on $U$. Then $N\le \dim_HU$.
\end{prop}
\begin{proof}
By restriction to a subset of the form $U\cap B(x_0,R)$, for $x_0\in X, R>0$, of positive measure, it suffices to assume that $U$ has finite measure. 
	The claim is automatic, if $\dim_HU=\infty$. Thus, assume that the Hausdorff dimension is finite. 
	
	Set $\nu = \varphi_*(\mu|_U)$ and let $V(\nu)$ be the decomposability bundle of $\nu$.  If $V(\nu)_x$ has dimension $N$ for almost every $x$ with respect to $\nu$, then $\nu\ll\lambda$ by Theorem \ref{thm:decompbundle}(1) and thus $\mathcal{H}^{N}(\varphi(U))>0$, since $\nu$ is concentrated on $\varphi(U)$. Then $N\leq \dim_H(\varphi(U)) \leq \dim_H(U)$.
	
	Suppose then to the contrary, that there exists a subset $A\subset U$ with positive $\nu$-measure where $V(\nu)_x$ has dimension less than $\dim_H(U)$ for each $x\in A$. 
	We can take $A$ to be Borel. Consider $\mu'= \mu|_{\varphi^{-1}(A)}$, which has push-forward $\nu' = \nu|_A=\varphi_*(\mu')$. By the third part in Theorem \ref{thm:decompbundle} we have that $V(\nu')_x = V(\nu)_x$ for $\nu'$-a.e. $x\in A$. Further $\varphi^{-1}(A) \subset U$, so $\varphi$ is still $p$-independent on $\varphi^{-1}(A)=U'$. 
	Now, by considering $U'$ instead of $U$ and $\nu'$ instead of $\nu$, we have that $V(\nu')_{\varphi(x)}$ has dimension less than $N$ for $\nu'$-almost every $x \in U$. In the following, we simplify notation by dropping the primes, and restricting to the positive measure subset $U'$ so constructed. 
	
	For $\nu$-almost every $x \in U$, we have that $V(\nu)_{\varphi(x)}$ is a strict subspace of $\R^N$, 
	and thus there are vectors perpendicular to these. Since $x\to V(\nu)_{\varphi(x)}$ is Borel, we can choose a Borel map $x\to \bm\xi_x\in (\R^N)^*$ so that $\bm\xi_x$ is a unit vector that vanishes on $V(\nu)_{\varphi(x)}$ for $\mu$-a.e. $x\in U$ (see e.g. \cite[Theorem 6.9.1]{bogachev07} which is an instance of a Borel selection theorem). Let $\tilde{U} \subset U$ be the full measure subset where these properties hold for every $x\in \tilde{U}$. Now, by Lemma \ref{lem:planinc} we have for $p$-a.e. curve $\gamma$ that $(\varphi\circ\gamma)'_t\in V(\nu)_{\varphi(\gamma_t)}$ for almost every $t\in \gamma\inv(U)$. The set $U \setminus \tilde{U}$ has null measure, and thus $\Gamma_{U \setminus \tilde{U}}^+$ has null modulus. 
	
	Thus, for $p$-a.e. curve $\gamma\in AC(I;X)$ and a.e. $t\in \gamma\inv(U)$ we can further assume $\gamma_t\in U$ or $|\gamma_t'|=0$. Therefore, $\bm\xi_{\gamma_t}((\varphi\circ\gamma)'_t)=0$ for almost every $t \in \gamma\inv(U)$ and such curves $\gamma$. By part (2) of Lemma \ref{lem:properties}, we have that $I(\varphi)  \leq \Phi^x(\bm\xi_x)=0$ for $\mu$-a.e. $x\in U$. This contradicts $p$-independence and proves the claim.
	
	
\end{proof}

\subsection{Sobolev charts}\label{sec:sobchart} By definition, a $p$-weak chart is a Lipschitz map which has target of maximal dimension with respect to Lipschitz maps. The notions of $p$-independence and maximality however are well-defined for any Sobolev map, and in fact $p$-weak charts \emph{could} be required to have Sobolev (instead of Lipschitz) regularity. Despite the apparent difference of the alternative definition, the existence of maximal $p$-independent Sobolev maps also guarantees the existence of $p$-weak chart of the same dimension. This follows from the energy density of Lipschitz functions, see \cite{seb2020}, together with results of the previous subsection.

\begin{prop}\label{prop:sobtolipchart}
	Suppose $p\ge 1$, and $\varphi\in N^{1,p}(X;\R^N)$ is $p$-independent and $p$-maximal in a bounded Borel set $U\subset X$. For any $\varepsilon>0$ there exists $V\subset U$ with $\mu(U\setminus V)<\varepsilon$, and a Lipschitz function $\psi:X\to \R^N$ such that $(V,\psi)$ is an $N$-dimensional $p$-weak chart.
\end{prop}
\begin{proof}
For any $V\subset U$ with $\mu(V)>0$, let $n_V$ be the supremum of numbers $n$ so that there exists $\psi\in \LIP_b(X;\R^n)$ which is $p$-independent on a positive measure subset of $V$. By the maximality of $N$ we have that $n_V\le N$. Thus $n_V$ is attained for every such $V$ and, by \cite[Proposition 3.1]{keith04}, there is a partition of $U$ up to a null-set by $p$-weak charts $V_i, i\in \N$ of dimension $\le N$. By \cite[Theorem 1.1]{seb2020}, Corollary \ref{cor:localnormdensity} (with a diagonal argument) and Mazur's lemma we have that, for each component $\varphi_k\in N^{1,p}(X)$ of $\varphi$, there exists a sequence $(\psi_k^n)\subset \LIP_b(X)$ with $|D(\varphi_k-\psi_k^n)|_p\to 0$ in $L^p(V_i)$. Thus, $|D(\varphi_k-\psi_k^n)|_p\to 0$ in $L^p(U)$. Here, we use that $|D(\varphi_k-\psi_k^n)|_p\leq |D\varphi_k|_p+|D\psi_k^n|_p$ and the $L^p$-convergence of the right hand side from \cite{seb2020}. 

If $\Phi$ and $\Psi_n$ denote the canonical minimal gradients associated to $\varphi$ and $\psi^n:=(\psi_1^n,\ldots,\psi_N^n)$ we have that
\begin{align*}
\sup_{\|\xi\|_*=1}|\Phi(\xi,\cdot)-\Psi_n(\xi,\cdot)|\le \underset{\|\xi\|_*=1}\esssup |D(\xi\circ(\varphi-\psi^n))|_p\le \sum_{k=1}^N|D(\varphi_k-\psi_k^n)|_p \quad \mu-\textrm{a.e. in }U.
\end{align*}
It follows that
\[
\lim_{n\to\infty}\mu(U\setminus \{ I(\psi^n)>0 \})=0,
\]
completing the proof, since $\psi^n$ is $p$-independent and maximal on the set $ \{ I(\psi^n)>0 \}$.
\end{proof}

Another condition in this context is strong maximality: 
a map $\varphi\in N^{1,p}(X;\R^N)$ is \emph{strongly maximal} in $U\subset X$ if no positive measure subset $V\subset U$ admits a $p$-independent Sobolev map into a higher dimensional Euclidean space. This condition excludes not only Lipschitz, but also Sobolev functions into higher dimensional targets, and is thus a priori stronger than maximality. However, it follows from Proposition \ref{prop:sobtolipchart} that a maximal $p$-independent Sobolev map is also strongly maximal. Conversely, if one has a Lipschitz chart, then the Lipschitz chart is also strongly maximal.

\subsection{$p$-weak charts in Poincar\'e spaces}\label{sec:poincare-diff} Recall that a metric measure space $X=(X,d,\mu)$ is said to be a $p$-PI space if $\mu$ is doubling, and $X$ supports a weak $(1,p)$-Poincar\'e inequality: there exist constants $C,\sigma>0$ so that, for any $f\in L^1(X)$ with upper gradient $g$, we have
\begin{align*}
\dashint_B|f-f_B|\ud\mu\le Cr\left(\dashint_{\sigma B}g^p\ud\mu\right)^{1/p}
\end{align*}
for all balls $B\subset X$ of radius $r$. Here $h_B=\dashint_B h\ud\mu=\frac{1}{\mu(B)}\int_Bh\ud\mu$ for a ball $B\subset X$ and $h\in L^1(B)$. Cheeger's celebrated result, from \cite{che99}, states that a PI-space admits a Lipschitz differentiable structure. We will return to this structure in Section \ref{sec:lipdiff}, but here recall the constructions from \cite[Section 4]{che99}. Cheeger's paper does not employ the following terminology,  but it simplifies and clarifies our presentation.

Given a Lipschitz map $\varphi:X \to \R^N$ and a positive measure subset $U \subset X$ the pair $(U,\varphi)$ is called \emph{a Cheeger chart} if for every Lipschitz map $f:X \to \R$ and a.e. $x\in U$ there is a unique element $\ud_{C,x} f\in (\R^N)^*$ satisfying 
\begin{equation}\label{eq:cheeger-Lip}
\Lip (\ud_{C,x} f\circ \varphi-f)(x)=0.
\end{equation}
This equality is equivalent to Equation \eqref{eq:cheegerdiff}.

\begin{proof}[Proof of Theorem \ref{thm:PI}]
Let $(U,\varphi)$ be a $p$-weak chart of dimension $N$ and let $f\in \LIP(X)$. Denote by $\Phi$ the canonical minimal gradient of $(\varphi,f):X\to\R^{N+1}$, cf. Lemma \ref{lem:canonicalplan}. Since $X$ is a $p$-PI space, it follows that $\Lip h=|Dh|_p$ $\mu$-a.e  for any $h\in \LIP(X)$, see \cite[Theorem 6.1]{che99}. (In fact, the slightly easier comparability from \cite[Lemma 4.35]{che99} suffices for the following.) Then, for any $\xi \in (\R^N)^*$ and for $\mu$-a.e. $x\in U$, we have
\begin{align*}
\Lip(\xi\circ\varphi-f)(x)=\Phi^x(\xi,-1),\quad \xi\in (\R^N)^*.
\end{align*}
Arguing using in the proof of Lemmas \ref{lem:canonicalplan} and \ref{lem:difflip} we obtain this equality, simultaneously, for a.e. $x\in U$ and for any $\xi \in A$ for a dense subset of $A \subset (\R^N)^*$. From this, and the continuity of both sides in $\xi$, we obtain that for $\mu$-a.e. $x\in U$, the equality holds simultaneously for all $\xi\in(\R^N)^*$.

Since the $p$-weak differential $\ud f$ is characterized by the property $\Phi^x(\ud f,-1)=0$ for $\mu$-a.e. $x\in U$, it follows that for $\mu$-a.e. $x\in U$, $\ud_xf\in (\R^N)^*$ satisfies Equation \eqref{eq:cheeger-Lip}. Thus $(U,\varphi)$ is a Cheeger chart. The uniqueness follows from the equality in a similar way.
\end{proof}
\begin{remark}\label{rmk:pweak-to-cheeger}
	The proof of Theorem \ref{thm:PI} also yields the claim under the weaker assumption $\Lip f\le \omega(|Df|_p)$ for some collection of moduli of continuity $\omega$ (compare Theorem \ref{thm:bundles}) since the equality $\Lip f =|Df|_p$ follows from this by \cite[Theorem 1.1]{tet2020}.
\end{remark}

\section{The $p$-weak differentiable structure}\label{sec:differentialstruct}

\subsection{The $p$-weak cotangent bundle}

A \emph{ measurable $L^\infty$-bundle} $\mathcal{T}$ over $X$ consists of a collection $(\{U_i,V_{i,x}\})_{i\in I}$ together with a collection $(\{\phi_{i,j,x}\} )$ of transformations with a countable index set $I$, where, 
\begin{enumerate}
	\item $U_i\subset X$ are Borel sets for each $i\in I$, and cover $X$ up to a $\mu$-null set;
	\item for any $i\in I$ and $\mu$-a.e. $x\in U_i$, $V_{i,x}=(V_i,|\cdot|_{i,x})$ is a finite dimensional normed space so that $x\mapsto |v|_{i,x}$ is Borel for any $v\in V_i$;
	\item for any $i,j\in I$ and $\mu$-a.e. $x\in U_i\cap U_j$, $\phi_{i,j,x}\colon V_{i,x} \to V_{j,x}$ is an isometric bijective linear map satisfying the \emph{cocycle condition}: for any $i,j,k \in I$ and $\mu$-a.e. $x \in U_i \cap U_j \cap U_k$, we have $\phi_{j,k,x} \circ \phi_{i,j,x}=\phi_{i,k,x}$.
\end{enumerate}

For each $i\in I$ and $\mu$-a.e. $x\in U_i$, we denote $\mathcal T_x$ the equivalence class of the normed vector space $V_{i,x}$ under identification by isometric isomorphisms. By (3), $\mathcal T_x$ is well-defined for $\mu$-a.e. $x\in X$.

We now show that a $p$-weak differentiable structure $\mathscr A$ on $X$ gives rise to a measurable bundle.

\begin{prop}\label{prop:measbundle}
	Let $p\ge 1$, and let $\{(U_i,\varphi_i)\}$ be an atlas of $p$-weak charts on $X$. The collection $\{(U_i,(\R^{N_i})^*,|\cdot|_{i,x}) \}$ forms a measurable bundle over $X$, the transformations given by the collection $\{ D\Phi_{i,j,x}\}$ constructed in Lemma \ref{lem:transformation}.
\end{prop}

First, we construct the transformation maps.

\begin{lemma}\label{lem:transformation}
	Let $(U_i,\varphi^i)$ be $N_i$-dimensional $p$-weak charts on $X$, with corresponding differentials $\ud^i$ and norms $|\cdot|_{i,x}$, for $i=1,2$. If $\mu(U_1\cap U_2)>0$, then $N_1=N_2:=N$ and, for $\mu$-a.e. $x\in U_1\cap U_2$, there exists a unique bijective isometric isomorphism $D\Phi_{1,2,x}:((\R^N)^*,|\cdot|_{1,x})\to ((\R^N)^*,|\cdot|_{2,x})$ such that $\ud^1f=\ud^2f\circ D\Phi_{1,2,x}$. Further $D\Phi_{1,2,x}$ satisfies the measurability constraint (2) for
\end{lemma}
In the proof, we denote by $\varphi_1^i,\ldots, \varphi_{N_i}^i$ the components of $\varphi^i$.
\begin{proof}
	For $\mu$-a.e. $x\in U_1\cap U_2$, define
	\begin{align*}
	D_x=D=(\ud^1\varphi_1^1,\ldots,\ud^2\varphi_{N_1}^1):\R^{N_2}\to \R^{N_1}.
	\end{align*}
	$D$ is a linear map satisfying, for all $\xi\in (\R^{N_1})^*$
	\begin{align}\label{eq:transformation}
	\xi\circ D((\varphi^2\circ\gamma)_t')=\xi((\varphi^1\circ\gamma)_t')\quad a.e.\ t\in \gamma\inv(U_1\cap U_2)
	\end{align}
	for $p$-a.e. $\gamma\in \Gamma_{U_1\cap U_2}^+$. Note that, by the uniqueness of differentials, $D$ is the unique linear map satisfying \eqref{eq:transformation} for $p$-a.e. curve. By Lemma \ref{lem:properties}(2) it follows that 
	\begin{align*}
	|\xi\circ D|_{2,x}=|\xi|_{1,x},\quad \xi\in (\R^{N_1})^*
	\end{align*}
	for $\mu$-a.e. $x\in U_1\cap U_2$. Thus $D^*$ is an isometric embedding and in particular $N_1\le N_2$. Reversing the roles of $\varphi^1$ and $\varphi^2$ we obtain that $N_1=N_2$ and consequently $D\Phi_{1,2,x}:=D^*_x:((\R^{N_1})^*,|\cdot|_{1,x})\to ((\R^{N_2})^*,|\cdot|_{2,x})$ is an isometric isomorphism for $\mu$-a.e. $x\in U_i\cap U_j$.
	
	For any $f\in \Ne pX$, the identity $\ud^1_xf=\ud^2_x f\circ D\Phi_{1,2,x}$ for $\mu$-a.e. $x\in U_1\cap U_2$ follows from \eqref{eq:transformation} and \eqref{eq:differential}.

\end{proof}

\begin{proof}[Proof of Proposition \ref{prop:measbundle}]
	Conditions (1) and (2) are satisfied by Lemma \ref{lem:representative}. The cocycle condition follows from Lemma \ref{lem:transformation}. 
\end{proof}
\begin{defn}
	We call the measurable bundle given by Proposition \ref{prop:measbundle} the $p$-weak cotangent bundle and denote it by $T^*_pX$. We denote $T_{p,x}^*X=((\R^N)^*,|\cdot|_x)$ and $T_{p,x}X=(\R^N,|\cdot|_{\ast,x})$ for almost every $x\in U$ where $(U,\varphi)$ is an $N$-dimensional $p$-weak chart and $|\cdot|_x$ the norm given by the canonical minimal gradient $\Phi$, cf. Lemmas \ref{lem:canonicalplan} and \ref{lem:difflip}.  
	The spaces $T_{p,x}$ are here defined pointwise almost everywhere. By considering the adjoints of transition maps in the definition above, one can patch these together to form a measurable $L^\infty$ tangent bundle, which is dual to $T^*_p X$, whose fibers are $T_{p,x}X$.
\end{defn}

The next proposition establishes the existence of a $p$-weak differentiable structure under a mild finite dimensionality condition.
\begin{prop}\label{prop:exists-p-weak-diff-str}
	Suppose $X$ is a metric measure space and $\{X_i\}_{i\in\N}$ a covering of $X$ with $\dim_HX_i<\infty$. Then, for any $p\ge 1$, $X$ admits a $p$-weak differentiable structure. Moreover, $N\le \dim_HX_i$ whenever $(U,\varphi)$ is an $N$-dimensional $p$-weak chart with $\mu(U\cap X_i)>0$.
\end{prop}
\begin{proof}
	For any Borel set $U\subset X$ with $\mu(U)>0$ there exists $i\in \N$ such that $\mu(U\cap X_i)>0$. By Proposition \ref{prop:chahausd} we have that $N\le \dim_H(U\cap X_i)$ whenever $\varphi\in \LIP_b(X;\R^N)$ is $p$-independent in a positive measure subset of $U\cap X_i$. Using \cite[Proposition 3.1]{kei04} we can cover $X$ up to a null-set by Borel sets $U_k$ for which there exists $\varphi_k\in\LIP_b(X;\R^{N_k})$ that are $p$-independent and $p$-maximal on $U_k$. The collection $\{(U_k,\varphi_k)\}_{k\in \N}$ is a $p$-weak differentiable structure on $X$. The last claim follows by the argument above.
\end{proof}

\subsection{Sections of measurable bundles}
A  measurable bundle $\mathcal T$ over $X$ comes with a projection map $\pi:\mathcal T\to X,\quad (x,v)\mapsto x$, and a {\bf section} of $\mathcal T$ is a collection $\omega = \{\omega_i\colon U_i \to V_i\}$ of Borel measurable maps satisfying $\pi\circ \omega_i={\rm id}_{U_i}$ $\mu$-a.e. and $\phi_{i,j,x}(\omega_i)=\omega_j$ for each $i,j \in I$ and almost every $x \in U_i\cap U_j$. Observe that the map $x\mapsto |\omega(x)|_x$ given by
\begin{align}\label{eq:ptwisenorm}
|\omega(x)|_x:=|\omega_i(x)|_{i,x} \quad \mu-\textrm{a.e. }x\in U_i
\end{align}
is well-defined up to negligible sets by the cocycle condition and the fact that $\phi_{i,j,x}$ is isometric.

\begin{defn}\label{def:bundlemodule}
For  $p\in [1,\infty]$, let $\Gamma_p(\mathcal T)$ be the space of sections $\omega$ of $\mathcal T$ with
\begin{align*}
\|\omega\|_p:=\|x\mapsto |\omega(x)|_x\|_{L^p(\mu)}<\infty.
\end{align*}
We call $\Gamma_p(\mathcal T)$ the space of $p$-integrable sections of $\mathcal T$. The space $\Gamma_p(T_p^*X)$ is called the $p$-weak cotangent module.
\end{defn}

Note that $\Gamma_p(\mathcal T)$, equipped with the pointwise norm \eqref{eq:ptwisenorm} and the natural additition and multiplication operations, is a normed module in the sense of \cite{gig15}. Recall that an $L^p$-normed $L^\infty$-module over $X$ is a Banach module $(\mathscr M,\|\cdot\|)$ over $L^\infty(X)$, equipped with a pointwise norm $|\cdot|:X\to \R$ that satisfies
\begin{align*}
|gm|=|g||m|\quad\textrm{and}\quad\|m\|=\left(\int_X|m|_x^p\ud\mu(x)\right)^{1/p}
\end{align*}
for all $m\in \mathscr M$ and $g\in L^\infty(X)$. We refer to \cite{gig15,gig18} for a detailed account of the theory of normed modules.


Next we consider the $p$-weak cotangent module $\Gamma_p(T^*_pX)$. For a $p$-weak chart $(U,\varphi)$ of $X$ and $f\in \Ne pX$, denote by $\ud_{(U,\varphi)}f$ the differential of $f$ with respect to $(U,\varphi)$. Lemma \ref{lem:transformation} implies that the collection of differentials with respect to different charts satisfies the compatibility condition above.
\begin{defn}\label{def:diffsection}
	Let $p\ge 1$, and suppose $\mathscr A$ is a $p$-weak differentiable atlas of $X$. For any $f\in \Ne pX$, the differential $\ud f\in \Gamma_p(T_p^*X)$ is the element in the $p$-weak cotangent module defined by the collection $\{\ud_{(U,\varphi)} f:U\to(\R^N)^*\}_{(U,\varphi)\in \mathscr A}$. 
\end{defn}

We record the following properties of the differential. The claims follow directly from Proposition \ref{prop:calcrules} together with the compatibility condition of sections. We omit the proof. 

\begin{prop}\label{prop:sectioncalcrules}
Let $A\subset X$ be a Borel set and $F:X\to Y$ a Lipschitz map to a metric measure space $(Y,d,\nu)$ admitting a $p$-weak differentiable structure, with $F_\ast\mu\le C\nu$.
\begin{enumerate}  
\item If $f,g\in \Ne pX$ agree on $A\subset X$, then $\ud f=\ud g$ $\mu$-a.e. on $A$.

\item If $f,g\in \Ne pX \cap L^\infty(X)$, then $\ud (fg)=g\ud f+f\ud g$ $\mu$-a.e.

\item If $E$ is the set of $y\in Y$for which $T_{p,y}^*Y$ does not exist, then $\mu(F\inv(E))=0$ and, for $\mu$-a.e. $x\in X\setminus F\inv(E)$ there exists a unique linear map $D_xF:T_{p,x}X\to T_{p,F(x)}Y$ such that
$$
\ud_x(h\circ F)=\ud_{F(x)} h\circ D_xF\quad\mu-\textrm{a.e. }x
$$
for every $h\in \Ne pY$.
\end{enumerate}
\end{prop}

We finish the subsection with a proof of the density of Lipschitz functions in Newtonian spaces.

\begin{proof}[Proof of Theorem \ref{thm:lipdense}]
	Let $f\in N^{1,p}(X)$. By \cite{seb2020}, there exists a sequence $(f_n)\subset\LIP_b(X)$ with $f_n\to f$ and $|Df_n|_p\to |Df|_p$ in $L^p(\mu)$. It follows that $(\ud f_n)\subset \Gamma_p(T^*X)$ is equi-integrable, and Remark \ref{rmk:equi-int}, Lemma \ref{lem:diff} together with a diagonalization argument over a union of charts covering $X$ shows that $\ud \tilde f_n\to \ud f$ in $\Gamma_p(T^*X)$ for convex combinations $\tilde f_n \in \LIP_b(X)$ of $f_n$'s. Consequently $|D(\tilde f_n-f)|_p\to 0$ in $L^p(\mu)$.
\end{proof}

\subsection{Dependence of the $p$-weak differentiable structures on $p$}\label{sec:relationship}

Suppose $1\le p<q$. We have that $|Df|_p\le |Df|_q$ $\mu$-a.e. for every $f\in \LIP_b(X)$, and the inequality may be strict, see \cite{garethdimarino}. As a consequence, if $\varphi\in \LIP_b(X;\R^N)$ is $q$-maximal in $U\subset X$, then it is $p$-maximal. It follows (using this dimension upper bound and \cite[Proposition 3.1]{kei04}) that if $X$ admits a $q$-weak differentiable structure then $X$ also admits a $p$-weak differentiable structure. We remark that the structures may be different.

For the following statement we say that a \emph{bundle map} $\pi:\mathcal T\to\mathcal T'$ between two measurable bundles $\mathcal T=(\{U_i,V_{i,x}\},\{\phi_{i,l,x}\})_{i\in I}$ and $\mathcal T'=(\{U_j',V_{j,x}'\},\{\psi_{j,k,x}\})_{j\in J}$ over $X$ is a collection of linear maps $\{\pi_{i,j,x}:V_i\to V_j'\}$ for $\mu$-a.e.  $x\in U_i\cap U'_j$ such that
\begin{itemize}
	\item[(a)] for each $i\in I$, $j\in J$, the map $x\mapsto \pi_{i,j,x}(v):U_i\cap U_j'\to V_j'$ is Borel for any $v\in V_i$; 
	\item[(b)] for each $i,l\in I$, $j,k\in J$, and $\mu$-a.e. $x\in U_i\cap U_l\cap U_j'\cap U_k'$, we have the compatibility condition: $\psi_{j,k,x}\circ\pi_{i,j,x}=\phi_{l,j,x}\circ\pi_{i,l,x}$.
\end{itemize}

When the underlying index sets agree and $U_i=V_i$ for all $i\in I$, it is sufficient to consider the family $\{\pi_{i,x}:=\pi_{i,i,x}\}$, since these determine a unique bundle map.

\begin{prop}\label{prop:pq}
	Suppose $q> p\ge 1$ and $X$ admits a $q$-weak differentiable structure. Then $X$ admits $p$-weak differentiable structure and there is a bundle map $\pi_{p,q}:T^*_qX\to T^*_pX$ which is a linear 1-Lipschitz surjection $\mu$-a.e. Moreover, this map satisfies $\pi_{p,q}=\pi_{p,s}\circ\pi_{s,q}$ for $q>s>p$, and $\pi_{p,q}(\ud_q f) = \ud_p f$ for any $f\in \LIP_b(X)$ where $\ud_q f, \ud_pf$ are the $p$- and $q$-weak differentials respectively.
\end{prop}

\begin{proof}
Since $X$ admits a $q$-differential structure, we can find  $q$-charts $(U_i,\varphi_{q,i})$ so that $X = \bigcup_{i \in \N} U_i \cup N$ with $\mu(N)=0$, and $\varphi_{q,i}\in N^{1,p}(X;\R^{m_i})$ is Lipschitz. Assume that $U_i$ are chosen to be pairwise disjoint. As $|Df|_p\le |Df|_q$ (a.e.) for any $f \in \LIP_b(X)$, any $p$-independent map is also $q$-independent. 
Any map $\varphi\in N^{1,p}(X;\R^n)$ which is $p$-independent on some positive measure subset of $U_i$ must have $n \leq m_i$, see Proposition \ref{prop:sobtolipchart}. By \cite[Proposition 3.1]{kei04} and this dimension bound we can cover $X$ by maximal $p$-independent maps, i.e. charts, $(V_j, \varphi_{p,j})$. By considering the countable collection of sets $V_i \cap U_j$, and re-indexing, we may assume that $(U_i,\varphi_{q,i})$ and $(U_i,\varphi_{p,i})$ are $q$- and $p$- charts, respectively.

We define the matrix $A_x$ for $x \in U_i$, by taking as rows the vectors $\ud_{i,p}\varphi^k_{q,i}$ for each component $k=1,\dots, m_i$. 
We define the bundle map $\pi_{p,q}$ by setting $\pi_{p,q}^x(\xi)=\xi \circ A_x$ for $\mu$-a.e. $x\in U_i$. For each $\xi$ we get that $\ud_p(\xi\circ \varphi_{q,i}) = \xi \circ A_x$. Thus, for $p$-a.e. curve $\gamma \in AC(I;X)$ and a.e. $t\in \gamma\inv(U)$ we have

$$\xi(\varphi_{q,i} \circ \gamma)'_t = (\xi \circ A_x)(\varphi_{p,i} \circ \gamma)_t'.$$

By the definition of the differential, we get immediately that $\pi_{p,q}(\ud_{q} f) = \ud_{p} f$ for every $f \in \LIP_b(X)$. Thus, the 1-Lipschitz property follows immediately from the definition of norms combined with $|Df|_p\le |Df|_q$. The map is clearly a surjective bundle map as well, and by uniqueness of the $p$-differential, we automatically get $\pi_{p,s}\circ\pi_{s,q}= \pi_{p,q}$.
\end{proof}

\section{Relationship Cheeger's and Gigli's differentiable structures} \label{sec:connection-gigli-cheeger}

\subsection{Gigli's cotangent module}

Fix $p\ge 1$. Gigli's cotangent module is the $L^p$-normed $L^\infty$-module given by the following theorem.

\begin{thm}\label{thm:cotangmodule}
	There exists an $L^p$-normed $L^\infty$-module $L^p(T^*X)$, with pointwise norm denoted $|\cdot|_G$, and a bounded linear map $\ud_G:\Ne pX\to L^p(T^*X)$ satisfying
	\begin{align}\label{eq:giglidiff}
	|\ud_Gf|_G=|Df|_p,\quad f\in \Ne pX,
	\end{align}
	such that the subspace $\mathcal{V}$ defined by
	\begin{align*}
	\mathcal V:= \left\{ \sum_j^M\chi_{A_j}\ud_Gf_j:\ (A_j)_j\textrm{ Borel partition of }X,\ f_j\in \Ne pX \right\}
	\end{align*}
	is dense in $L^p(T^*X)$. The module $L^p(T^*X)$ is uniquely determined up to isometric isomorphism of normed modules by these properties.
\end{thm}

Following \cite[Definition 1.4.1]{gig18} we say that a collection $\{v_1,\ldots,v_N\}\subset L^p(T^*X)$ is \emph{linearly independent} in a Borel set $U\subset X$ if, whenever $g_1,\ldots,g_N\in L^\infty(X)$ satisfy $\left|\sum_j^Ng_jv_j\right|_G=0$ $\mu$-a.e. on $U$, we have that $g_1=\cdots=g_N=0$ $\mu$-a.e. in $U$. A linearly independent collection $\{v_1,\ldots,v_N\}$ in $U$ is a \emph{basis} of $L^p(T^*X)$ in $U$ if, for any $v\in L^p(T^*X)$ there exists a Borel partition $\{U_i\}_{i\in\N}$ of $U$ and $g_1^i,\ldots,g_N^i\in L^\infty(X)$ such that
$\left|v- \sum_j^Ng_j^iv_j\right|_G=0$ $\mu$-a.e. on $U_i$, for every $i\in \N$.

\begin{defn}
	Let $p\ge 1$. The cotangent module $L^p(T^*X)$ is locally finitely generated if there exists a Borel partition such that $L^p(T^*X)$  has a finite basis in each set of the partition.
\end{defn}
By \cite[Proposition 1.4.5]{gig18}, there exists a Borel partition $\{A_N\}_{N\in\N\cup\{\infty\}}$ of of $X$ such that $L^p(T^*X)$ has a basis of $N$ elements on $A_N$, for each $N\in \N\cup\{\infty\}$. We call the partition $\{A_N\}$ the \emph{dimensional decomposition} of $X$. Notice that $L^p(T^*X)$ is locally finitely generated if and only if $\mu(A_\infty)=0$. 

In the forthcoming discussion we identify vectors (and vector fields) $\xi\in \R^N$ with their dual element $v\mapsto v\cdot\xi$ where necessary.
\begin{lemma}\label{lem:isometry}
	Let $p\ge 1$, $N\ge 0$, $\varphi=(\varphi_1,\ldots,\varphi_N)\in \Ne pX^N$, and $\Phi$ the canonical minimal gradient associated to $\varphi$. If $\bm g=(g_1,\ldots,g_N)\in L^\infty(X;(\R^N)^*)$, then
	\begin{align*}
	\left|\sum_{k=1}^N g_k\ud_G\varphi_k\right|_{G,x}=\Phi^x(\bm g)\quad \mu-\textrm{a.e. }x\in X.
	\end{align*}
	In particular, $\varphi$ is $p$-independent on $U\subset X$ if and only if $\ud_G\varphi_1,\ldots,\ud_G\varphi_N\in L^p(T^*X)$ are linearly independent on $U$.
\end{lemma}
\begin{proof}
	If $g_1,\ldots,g_N$ are simple functions, then $\bm g=\sum_j^M\chi_{A_j}\xi_j$ for disjoint Borel $A_j$ and some $\xi_j\in (\R^{N})^*$. It follows that $\sum_{k=1}^N g_k\ud_G\varphi_k= \sum_j^M\chi_{A_j}\ud_G(\xi_j\circ\varphi)$ as elements of $L^p(T^*X)$. Thus 
	\begin{align*}
	\left|\sum_{k=1}^N g_k\ud_G\varphi_k\right|_x=\left|\sum_j^M\chi_{A_j}\ud_G(\xi_j\circ\varphi)\right|_x=\sum_j^M\chi_{A_j}|D(\xi_j\circ\varphi)|_p=\Phi^x(\bm g)
	\end{align*}
	for $\mu$-a.e. $x\in X$.
	
	The estimate
	\begin{align*}
	\Phi^x(\bm g)\le \left(\sum_k^N|g_k|^q\right)^{1/q}\left(\sum_k^N|D\varphi_k|^p_p\right)^{1/p}\le C|\bm g|\sum_k^N|D\varphi_k|_p,
	\end{align*}
	valid for all simple vector valued $\bm g$, implies that the equality in the claim is stable under local $L^\infty$-convergence of $\bm g$. Since simple functions are dense in $L^\infty$, the claim follows. The remaining claim follows in a straightforward way from the equality.
\end{proof}

\begin{remark}\label{rmk:differential-G} If $\varphi \in \LIP(X;\R^N)$ is a chart in $U$, and $f \in N^{1,p}(X)$, then for the canonical minimal upper gradient $\Phi^x(a, \bm \xi))$ of $(f,\varphi)\in N^{1,p}_{loc}(X;\R^{N+1})$ we have by Lemma \ref{lem:properties}(2) that $\Phi^x(1, -\ud f) = 0$. Thus, by the previous lemma, we get that $\ud_G f - \sum_{k=1}^Ng^k \ud_G\varphi_k=0$, where $g^k$ are the components of $\ud f$ and $\varphi_k$ are the components of $\varphi$.  Indeed, this follows by considering this first on the sets $U_M=\{x \in U: |g^k(x)| \leq M, k=1, \dots, N\}$ and sending $M \to \infty$ combined with locality.
\end{remark}

\begin{lemma}\label{lem:indep}
	If $(U,\varphi)$ is an $N$-dimensional $p$-weak chart in $X$, then the differentials of the component functions $\ud_G\varphi_1,\ldots,\ud_G\varphi_N$ form a basis of $L^p(T^*X)$ in $U$.
\end{lemma}
\begin{proof}
	By Lemma \ref{lem:isometry}, $\ud_G\varphi_1,\ldots,\ud_G\varphi_N\in L^p(T^*X)$ are linearly independent on $U$. To see that they span $L^p(T^*X)$ in $U$, let $f\in N^{1,p}(X)$, and set $g_k:=\ud f(e_k)$ for each $k=1,\ldots,N$, where $e_k$ is the standard basis of $\R^N$. Then, since $\ud\varphi_k = e^k$, where $e^k$ is the dual basis of $(\R^N)^*$, we get $\ud f=\sum_{k=1}^Ng_k\ud \varphi_k$.  
	Thus, by Remark \ref{rmk:differential-G} we have $\ud_Gf=\sum_{k=1}^Ng_k\ud_G\varphi_k.$ Since the abstract differentials $\ud_Gf$ span $L^p(T^*X)$, this completes the proof.
\end{proof}

\begin{lemma}\label{lem:compatible}
Suppose $p\ge 1$ and $X$ admits a $p$-weak differentiable structure. There exists an isometric isomorphism $\iota:\Gamma_p(T^*_pX)\to L^p(T^*X)$ of normed modules satisfying
\begin{align}\label{eq:compatible}
\iota(\ud f)=\ud_Gf, \quad f\in N^{1,p}(X).
\end{align}
The map $\iota$ is uniquely determined by \eqref{eq:compatible}.
\end{lemma}
Uniqueness here means that if $A:\Gamma_p(T^*_pX)\to L^p(T^*X)$ is $L^\infty$-linear and satisfies \eqref{eq:compatible} then $A=\iota$.
\begin{proof}
The set
\[
\mathcal W=\left\{ \sum_j^M\chi_{A_j}\ud f_j:\ (A_j)_j \textrm{ Borel partition of $X$, }f_j\in \Ne pX \right\}
\]
is dense in $\Gamma_p(T^*_pX)$, since it contains all the simple Borel sections of $T_p^*X$. We set
\begin{align*}
\iota(v):=\sum_j^M\chi_{A_j}\ud_Gf_j,\quad v=\sum_j^M\chi_{A_j}\ud f_j\in\mathcal W,
\end{align*}
We have that
$$
|\iota(v)|_G=\sum_j^M\chi_{A_j}|Df_j|_p=\sum_j^M\chi_{A_j}|\ud f_j|=|v|\quad\mu-a.e.,
$$
for $v\in \mathcal W$. This implies that $\iota$ is well-defined and preserves the pointwise norm on the dense set $\mathcal W$. By Remark \ref{rmk:differential-G} we have that $\iota$ is linear. Since $\iota(\mathcal W)=\mathcal V$, it follows that $\iota$ extends to an isometric isomorphism $\iota:\Gamma_p(T_p^*X)\to L^p(T^*X)$. Note that $\iota(\ud f)=\ud_Gf$ for every $f\in \Ne pX$, establishing \eqref{eq:compatible}.

To prove uniqueness, note that if $A:\Gamma_p(T_p^*X)\to L_p(T^*X)$ is linear and satisfies \eqref{eq:compatible}, then
$A(v)=\iota(v)$ for all $v\in \mathcal W$ which implies that $A=\iota$ by the density of $\mathcal W$.
\end{proof}

\begin{proof}[Proof of Theorem \ref{thm:gigli}]
	If $X$ admits a $p$-weak differentiable structure, Lemma \ref{lem:indep} implies that $L^p(T^*X)$ is locally finitely generated. To prove the converse implication, suppose $\{A_N\}_{N\in\N\cup\{\infty\}}$ is the dimensional decomposition of $X$ and $\mu(A_\infty)=0$.
	
	Let $N\in\N$ be such that $\mu(A_N)\ge \mu(V)>0$ for some Borel set $V$, and $v_1,\ldots,v_N\in L^p(T^*X)$ is a basis of $L^p(T^*X)$ on $V$. By possibly passing to a smaller subset of $V$, we may assume that there exists $C>0$ for which 
	\begin{equation}\label{eq:lowerboundbasis}
	\int_V\left|\sum_k^Ng_kv_k\right|_G^p\ud\mu\ge \frac 1C\int_V|\bm g|^p\ud\mu,\quad \bm g=(g_1,\ldots,g_N)\in L^\infty.
	\end{equation}
	
	For each $k=1,\ldots, N$ there are sequences 
	\begin{align*}
v_k^n=\sum_j^{M_k^n}\chi_{A_{j,k}^n}\ud_Gf_{j,k}^n
	\end{align*}
	with $\{ A_{j,k}^n \}_j$ a Borel partition of $X$ and $(f_j^n)\subset \Ne pX$ such that $v_k^n\to v_k$ in $L^p(T^*X)$ as $n\to\infty$, by the definition of $L^p(T^*X)$.
	
	We set $J^n=\{1,\ldots,M_1^n\}\times\cdots\times\{1,\ldots,M_N^n\}$ and define new partitions $A^n_{\bar \jmath}:=A^n_{j_1,1}\cap\cdots\cap A^n_{j_N,N}$ indexed by $\bar \jmath=(j_1,\ldots,j_N)\in J^n$. Then 
	\begin{align}\label{eq:limn}
	&v_k^n=\sum_{\bar \jmath\in J^n}\chi_{A_{\bar \jmath}^n}\ud_G(f^n_{j_k,k}),\quad \mu(V)=\sum_{\bar\jmath\in J^n}\mu(A_{\bar\jmath}^n\cap V),\textrm{ and }\nonumber\\
	&\lim_{n\to\infty}\int_X|v^n_k-v_k|^p_G\ud\mu=\lim_{n\to\infty}\sum_{\bar\jmath\in J^n}\int_{A^n_{\bar\jmath}}|\ud_G\varphi^{n,\bar\jmath}_k-v_k|_G^p\ud\mu=0
	\end{align}
	for all $n$ and $k=1,\ldots,N$. We claim that there exists $n$ so that $\varphi^{n,\bar\jmath}:=(f^n_{j_1,1},\ldots,f^n_{j_N,N})\in N^{1,p}(X;\R^N)$ is $p$-independent on a positive measure subset of $A^n_{\bar\jmath}\cap V$, for some $\bar\jmath \in J^n$.
	
	By \eqref{eq:lowerboundbasis} we have the inequality
	\begin{align*}
	\frac 1C\int_{A^n_{\bar\jmath}\cap V}|\bm g|^p\ud\mu\le &\int_{A^n_{\bar\jmath}\cap V}\left|\sum_k^Ng_kv_k\right|^p\ud\mu \\
	\le & C'\int_{A^n_{\bar\jmath}\cap V}\left|\sum_k^Ng_k\ud_G\varphi^{n,\bar\jmath}_k\right|_G^p\ud\mu+C'\int_{A^n_{\bar\jmath}\cap V}\left|\sum_k^Ng_k(\ud_G\varphi^{n,\bar\jmath}_k-v_k)\right|_G^p\ud\mu\\
	\le &C'\int_{A^n_{\bar\jmath}\cap V}\Phi_{n,\bar\jmath}(\bm g(x),x)\ud\mu+C''\int_{A^n_{\bar\jmath}\cap V}|\bm g|^p\left(\sum_k^N|\ud_G\varphi^{n,\bar\jmath}_k-v_k|_G^p\right)\ud\mu
	\end{align*}
	for all $\bm g=(g_1,\ldots,g_N)\in L^\infty$, where $\Phi_{n,\bar\jmath}$ is the canonical minimal gradient of $\varphi^{n,\bar\jmath}$ (cf. Lemma \ref{lem:isometry}). By \eqref{eq:limn} there exists $n\in \N$, $\bar\jmath\in J^n$ and a Borel set $U\subset A^n_{\bar\jmath}\cap V$ with $0<\mu(U)\le \mu(A^n_{\bar\jmath}\cap V)$ such that $\sum_k^N|\ud_G\varphi^{n,\bar\jmath}_k-v_k|_G^p<\varepsilon$ on $U$, where $C''\varepsilon<\frac{1}{2C}$. Thus 
	\begin{align*}
	\frac 1C\int_{U}|\bm g|^p\ud\mu\le C'\int_{U}\Phi_{n,\bar\jmath}(\bm g(x),x)\ud\mu+\frac{1}{2C}\int_{U}|\bm g|^p\ud\mu	
	\end{align*}
	for all $\bm g=(g_1,\ldots,g_N)\in L^\infty(U;\R^N)$ by extending $g$ by zero to $V \setminus U$. This readily implies that $I(\varphi^{n,\bar\jmath})>0$ a.e. in $U$, proving the $p$-independence of $\varphi^{n,\bar\jmath}$ in $U$. Note that $\varphi^{n,\bar\jmath}$ is also maximal, since the existence of a Lipschitz map on a positive measure subset of $U$ with a higher dimensional target would imply that the local dimension of $L^p(T^*X)$ in $V$ would be $>N$, cf. Lemma \ref{lem:isometry}. By Proposition \ref{prop:sobtolipchart}, $U$ contains a $N$-dimensional $p$-weak chart, and \cite[Proposition 3.1]{kei04} implies that $X$ admits a differentiable structure.

	The argument above shows that each $A_N$ with $\mu(A_N)>0$ can be covered up to a null-set by $N$-dimensional $p$-weak charts, proving (b), while (a) follows directly from Lemma \ref{lem:compatible}. Finally, (c) is implied by Proposition  \ref{prop:chahausd}.
	\end{proof}

Theorem \ref{thm:gigli} and \cite[Chapter 2]{gig18} immediately yield the following corollary.

\begin{cor}\label{cor:giglicor}
	Let $p\ge 1$ and suppose $X$ admits a $p$-weak differentiable structure.
	\begin{itemize}
		\item[(i)] If $p>1$, then $\Ne pX$ is reflexive.
		\item[(ii)] If $p=2$, then $\Ne 2X$ is infinitesimally Hilbertian if and only if, for $\mu$-a.e. $x\in X$, the pointwise norm $|\cdot|_x$ (cf. Theorem \ref{thm:differentiability}) is induced by an inner product.
	\end{itemize}
\qed
\end{cor}

\subsection{Lipschtiz differentiability spaces}\label{sec:lipdiff}

A space $X$ is said to be a Lipschitz differentiability space if it admits a Cheeger structure. Recall that a Cheeger structure is a countable collection of Cheeger charts $(U_i,\varphi_i)$, see Section \ref{sec:poincare-diff}, so that $\mu(X \setminus \bigcup_i U_i) = 0$. Following \cite[Section 4, p. 458]{che99}, we note that the differentials $\ud_{C,i} f$ of a Lipschitz function $f$ with respect to overlapping charts satisfy a co-cycle condition almost everywhere and the transition maps preserve the pointwise norm.  Thus, they define a measurable  $L^\infty$-bundle $T^*_C X$ called the measurable cotangent bundle.

Suppose now that $X$ admits a Cheeger structure. Denote by $T^*_CX$ the associated measurable cotangent bundle, and by 
\begin{align*}
|\xi|_{C,x}:=\Lip(\xi\circ\varphi)(x),\quad \xi\in (\R^N)^*
\end{align*}
the pointwise norm for $\mu$-a.e. $x\in U$, where $(U,\varphi)$ is an $N$-dimensional Cheeger chart of $X$.

Fix $p\ge 1$. Any Lipschitz differentiability space $X$ admits a $p$-weak differentiable structure. Indeed, the asymptotic doubling property of the measure (cf. \cite{bate2013differentiability}) implies, by \cite[Lemma 8.3]{bate12diff}, that $X$ decomposes into finite-dimensional pieces. The existence of the $p$-weak differentiable structure now follows from Proposition \ref{prop:exists-p-weak-diff-str}, and the associated measurable cotangent bundle is denoted $T^*_pX$. We have the following result from \cite[Theorem 3.4]{tet2020}:
\begin{thm}\label{thm:tet}
Let $p\ge 1$. There exists morphism $P:\Gamma_p(T^*_CX)\to L^p(T^*X)$ of normed modules such that 
	\begin{itemize}
		\item[(a)]	$P(\ud_Cf)=\ud_Gf$ for every $f\in \LIP(X)$;
		\item[(b)] $|P(\omega)|_G\le |\omega|_C$ for every $\omega\in \Gamma_p(T^*_CX)$; and
		\item[(c)] for every $w\in L^p(T^*X)$ there exists $\omega\in P\inv(w)$ with $|w|_G=|\omega|_c$.
	\end{itemize}
\end{thm}
\begin{remark}
The proof of \cite[Theorem 3.4]{tet2020} can be modified to cover the case $p=1$: the energy density of Lipschitz functions holds for $p=1$ by \cite{seb2020}, and equicontinuity can be used instead of $L^p$-boundedness to obtain the weakly convergent subsequence in the proof.

\end{remark}

\begin{proof}[Proof of Theorem \ref{thm:bundles}]
Arguing as in the proof of Proposition \ref{prop:pq} we may assume that $X$ has a Borel partition $\{U_i\}$ and Lipschitz maps $\varphi_p^i=(\varphi_{p,1}^i,\ldots,\varphi_{p,N_i}^i),\varphi_C^i=(\varphi_{C,1}^i,\ldots,\varphi_{C,M_i}^i)$ such that $(U_i,\varphi_p^i)$ is a $p$-weak chart and $(U_i,\varphi_C^i)$ is a Cheeger chart on $X$ (of possibly different dimensions $N_i$ and $M_i$) for each $i\in \N$. For each $i$ and $\mu$-a.e. $x\in U_i$ define
$$
\sigma_{i,x}=(\ud_{p,x}\varphi_{C,1}^i,\ldots,\ud_{p,x}\varphi_{C,M_i}^i):\R^{N_i}\to \R^{M_i}.
$$
It is easy to see that the collection $\{\pi_{i,x}=\sigma_{i,x}^*\}$ defines a bundle map $T_C^*X\to T_p^*X$ satisfying
$$
\ud_{p,x}f=\ud_{C,x}f\circ \sigma_x\quad \mu-a.e.\ x\in X.
$$
for every $f\in \Ne pX$. This proves Equation \eqref{eq:ch-pweak-compatible}. In particular, for each $i\in \N$ and $\xi\in (\R^{M_i})^*$ we have $\pi_{i,x}(\xi)=\xi\circ\sigma_{i,x}=\ud_{p,x}(\xi\circ\varphi_{C}^i),$ and consequently $$|\pi_{i,x}(\xi)|_x=|D(\xi\circ\varphi_C^i)|_p(x)\le \Lip(\xi\circ\varphi_C^i)(x)=|\xi|_{C,x},$$
for $\mu$-a.e. $x\in U_i$. Moreover, for any $\zeta\in (\R^{N_i})^*$ setting $\xi:=\ud_{C,x}(\zeta\circ\varphi_p^i)$, we have that $$\pi_{i,x}(\xi)=\ud_{C,x}(\zeta\circ\varphi_p^i)\circ\sigma_x=\ud_{p,x}(\zeta\circ\varphi_p^i)=\zeta,$$ proving that $\pi_{i,x}$ is surjective for $\mu$-a.e. $x\in U_i$.

To prove that $\pi_{i,x}$ is a submetry for $\mu$-a.e. $x\in U_i$, suppose to the contrary that there exist a Borel set $B\subset U_i$ with $0<\mu(B)<\infty$ such that $\pi_{i,x}$ is not a submetry for $x\in B$. Then there exists a Borel map $\bm\zeta:B\to (\R^{N_i})^*$ with $|\bm\zeta_x|_x=1$ and
\begin{align}\label{eq:contradiction}
|\bm\zeta_x|_x=1\quad\textrm{ and }\quad \inf_{\xi\in \pi_{i,x}\inv(\bm\zeta_x)}|\xi|_{C,x}>1\quad \textrm{ for $\mu$-a.e. }\ x\in B.
\end{align}

We derive a contradiction using Theorem \ref{thm:tet} and the isometric isomorphism $\iota:\Gamma_p(T_p^*X)\to L^p(T^*X)$ from Theorem \ref{thm:gigli}(a). We may view $\bm\zeta$ as an element of $\Gamma_p(T_p^*X)$ by extending it by zero outside $B$. Set $w:=\iota(\bm\zeta)\in L^p(T^*X)$. Then $|w|_G=\chi_B$. By Theorem \ref{thm:tet}(c) there exists $\omega\in \Gamma_p(T_C^*X)$ with $P(\omega)=w$ and $|\omega|_{C}=|w|_G=\chi_B$ $\mu$-a.e. However, since $\omega_x\in \pi_{i,x}\inv(\bm\zeta_x)$ for $\mu$-a.e. $x\in B$, we have $|\omega|_{C,x}\ge \inf_{\xi\in \pi_{i,x}\inv(\bm\zeta_x)}|\xi|_{C,x}>1$ for $\mu$-a.e. $x\in B$ by \eqref{eq:contradiction}, which is a contradiction. This completes the proof that $\pi_{i,x}$ is a submetry for $\mu$-a.e. $x\in U_i$.


If $\Lip f\le \omega(|Df|_p)$ holds for every $f\in \LIP_b(X)$, then by \cite[Theorem 1.1]{tet2020} we have that $|Df|_p=\Lip f$ $\mu$-a.e. for every $f\in \LIP_b(X)$. It follows that $p$-weak charts are Cheeger charts (cf. Theorem \ref{thm:PI} and Remark \ref{rmk:pweak-to-cheeger}) and that the pointwise norms agree $\mu$-almost everywhere. This implies that the maps $\pi_{i,x}$ are isometric bijections for $\mu$-a.e. $x$.
\end{proof}

\appendix
\section{General measure theory}

\subsection{Measurability questions}

Here we record a host of measurability statements that are needed throughout the paper. See \cite{gigenr,AGS08,bogachev07} for more details. Given $f\in \Ne pX$ and a Borel representative $g$ of $p$-weak upper gradient of $f$, we denote  

\begin{align*}
\Gamma(f):=&\{ \gamma\in AC(I;X): f\circ\gamma\in AC(I;\R) \},\\
\Gamma(f,g):=&\{ \gamma\in AC(I;X): g\textrm{ upper gradient of $f$ along }\gamma \}\subset \Gamma(f)
\end{align*}
and 
\begin{align*}
MD&=\{(\gamma,t) \in AC(I;X)\times I : |\gamma_t'| \textrm{ exists}\},\\
\Diff(f)&=\{(\gamma,t)\in AC(I;X)\times I: \gamma\in \Gamma(f),\ (f\circ \gamma)_t' \textrm{ and }|\gamma_t'|>0 \textrm{ exist}\},\\
\Diff(f,g)&=\{ (\gamma,t)\in \Diff(f): \gamma\in \Gamma(f,g),\ |(f\circ\gamma)_t'|\le g_f(\gamma_t)|\gamma_t'| \}.
\end{align*}

Also, let $\len(\gamma)$ be the length of a curve $\gamma$, if the curve is rectifiable, and otherwise infinity.  The function  ${\rm der}$ is defined by ${\rm der}(\gamma,t) \defeq |\gamma_t'| =  \lim_{h \to 0} \frac{d(\gamma_{t+h},\gamma_t)}{|h|}$, when the limit exists, and otherwise is infinity.

\begin{lemma}\label{lem:measurable-integral} 
	\begin{enumerate}
		\item The functions $\len: C(I;X) \to [0,\infty]$ and ${\rm der}:AC(I;X)\times I\to [0,\infty]$ are Borel measurable.
		\item If $g:X \to [0,\infty]$ is a Borel function, then $I:AC(I;X) \to \R$, given by $\gamma \mapsto \int_\gamma g \, \ud s$ or $\infty$ if the curve is not rectifiable, is Borel.
		\item If $H:AC(I;X)\times I\to [0,\infty]$ is Borel, then $I_H(\gamma):=\int_0^1H(\gamma,s)\ud s:AC(I;X)\to [0,\infty]$ is Borel.
		\item The set $MD$ is Borel, and the map $MD\to\R$ defined by $(\gamma,t)\to|\gamma_t'|$ is Borel.
	\end{enumerate} 
\end{lemma}
\begin{proof} \begin{enumerate}
		
		\item[(1)] The length function is a lower semicontinuous function with respect to uniform convergence, and thus is Borel. Fix $r,p \in \Q$ positive. Then define $A_{p,r}=\cup_{n \in \N} \cap_{q \in \Q \cap (-\frac{1}{n}, \frac{1}{n})} \{(\gamma,t) : |d(\gamma_{t+q},\gamma_t)-q p|<r|q|\}$, and thus Borel. The set $M$ where the metric derivative exists is of the form $\cap_{r \in \Q \cap (0,\infty)} \cup_{p \in \Q \cap (0,\infty)}A_{p,r}$. On this set we have $M \cap A_{p,r} = {\rm der}^{-1}(B(p,r))$ and thus ${\rm der}(\gamma,t)$ is Borel.
		
		\item[(2)]The claims for the integral function being Borel follow from a monotone family argument, and considering $g$ first a characteristic function of an open set and using lower semi-continuity of the integral in that case.
		
		 \item[(3)]If $H$ is a characteristic function of a product set $A \times B$, where $A$ and $B$ are open sets such that $A \subset C(I;X), B \subset I$, then the claim follows just as in Statement (2). Again, by a monotone family argument, we obtain the claim for all Borel measurable functions.
		 \item[(4)]Define for every $q\in \Q$ and $\varepsilon,h>0$ the sets $A(\varepsilon,q,h)$ and $B(\varepsilon,q)$ by
	\begin{align*}
	A(\varepsilon,q,h)&\defeq \left\{ (\gamma,t)\in C(I;X)\times I:\ \left|\frac{d(\gamma_{t+h},\gamma_t)}{|h|}-q\right|<\varepsilon \right\}\\
	B(\varepsilon,q)&\defeq \bigcup_{\delta\in \Q_+}\bigcap_{h\in (0,\delta)\cap\Q}A(\varepsilon,q,h).
	\end{align*}
 We note that $|\gamma_t'|$ exists if and only if $\displaystyle (\gamma,t)\in \bigcap_{j\in\N}\bigcup_{q\in\Q}B(2^{-j},q)=MD$. On the set $MD$, where the limit exists, we can write $|\gamma_t'|=\lim_{n\to\infty}n(d(\gamma_{t+n^{-1}},\gamma_t))$, which shows measurability.
		
		
	\end{enumerate}
\end{proof}

\begin{lemma}\label{lem:nullborel}
	Let $g$ be a Borel $p$-weak upper gradient of $f\in \Ne pX$. There exists a Borel set $\Gamma_0\subset AC(I;X)$ with $\Mod_p(\Gamma_0)=0$ such that $AC\setminus\Gamma_0\subset \Gamma(f,g)$.
	
	Suppose moreover that $f$ is Borel. Then the set $A:=\Gamma_0^c\times I\cap \Diff(f,g)$ is Borel, and $\bm\pi(A^c)=0$ whenever $\bm\pi=\mathcal L^1\times \bm\eta$ and $\bm\eta$ is a $q$-test plan.
	
	If $f$ is Lipschitz, and $g=\Lip[f]$, then we can choose $\Gamma_0=\emptyset$, and $\Diff(f,g)=\Diff(f)$ is Borel.
\end{lemma}
Note that we make no claims about the Borel measurability of the set $\Gamma(f,g)$.
\begin{proof}
We model the argument after \cite[Lemma 1.9]{pasqualetto20b}.
Since $\Mod_p(\Gamma(f,g)^c)=0$ there exists an $L^p$-integrable Borel function $\rho:X\to [0,\infty]$ with $\int_\gamma\rho \ud s=\infty$ for every $\gamma\notin \Gamma_{f,g}$. Then  $\Gamma_0:=\{ \gamma \in AC(I;X): \int_\gamma\rho \ud s=\infty\}\supset \Gamma_{f,g}^c$ is a Borel set, by Lemma \ref{lem:measurable-integral} and $\bm\eta(\Gamma_0)=0$ for every $q$-plan $\bm\eta$ (see Remark \ref{rmk:modnull}). If $f$ is Lipschitz, then $\Gamma(f,g) = AC(I;X)$. Thus, we can choose $\Gamma_0 = \emptyset$.
	
	For the second part assume $f\in \Ne pX$ is Borel,  and set
	\begin{align*}
	A(\varepsilon,q,h)&=\left\{ (\gamma,t)\in \Gamma_0^c\times I:\ \left|\frac{f(\gamma_{t+h})-f(\gamma_t)}{h}-q\right|<\varepsilon \right\}\\
	B(\varepsilon,q)&=\bigcup_{\delta\in \Q_+}\bigcap_{h\in (0,\delta)\cap\Q}A(\varepsilon,q,h)
	\end{align*}
	for each $q\in \Q$ and $\varepsilon,h>0$. It is easy to see that for each $\gamma\notin \Gamma_0$, $(f\circ\gamma)_t'$ exists if and only if $$\displaystyle (\gamma,t)\in \bigcap_{j\in\N}\bigcup_{q\in\Q}B(2^{-j},q)=:A.$$ Note that $A$ is a Borel set with $A\cap MD\subset \Diff(f)$. Moreover, $(\gamma,t)\mapsto (f\circ\gamma)_t'$ is Borel when restricted to $A\cap MD$
	
	Define the Borel function $H(\gamma,t)=(f\circ\gamma)_t'$ if $(\gamma,t)\in A\cap MD$ and $H=+\infty$ otherwise, and $G(\gamma,t)=|H|-g(\gamma_t)|\gamma_t'|$ (here we use the convention $\infty-\infty=\infty$). Then the set
	\[ \{ G\le 0 \}=\Gamma_0^c\times I\cap\Diff(f,g) \]
	is Borel.

	Set $N:=\{ G>0 \}$, suppose $\bm\eta$ is a $q$-test plan and $\bm\pi:=\mathcal L^1\times\bm\eta$. Note that
	\[
	N\subset \Gamma_0\times I\cup\{ (\gamma,t)\in\Gamma_0^c\times I:\ G(\gamma,t)>0 \}.
	\]
	But for all $\gamma\notin\Gamma_0$, we have that $G(\gamma,t)\le 0$ for $\mathcal L^1$-a.e. $t\in I$. Thus 
	\[
	\bm\pi(N)\le \bm\eta(\Gamma_0)+\int_{\Gamma_0^c}\int_0^1\chi_{\{G(\gamma,\cdot)>0\}}(t)\ud t\ud\bm\eta(\gamma)=0,
	\]
	finishing the proof of the second part.
\end{proof}

\begin{cor}\label{cor:borel}
	Every pointwise defined function $f\in \Ne pX$ has a Borel representative $\bar f\in \Ne pX$. Moreover, if $f\in \Ne pX$ and $g$ is a Borel $p$-weak upper gradient of $f$, there exists a Borel set $N\subset C(I;X)\times I$ with $N^c\subset \Diff(f,g)$ and $\bm\pi(N)=0$ whenever $\bm\pi=\mathcal L^1\times \bm\eta$, $\bm\eta$ a $q$-test plan. The map $(\gamma,t)\mapsto (f\circ\gamma)_t'$ if $(\gamma,t)\notin N$ and $+\infty$ otherwise is Borel. If $f$ is Lipschitz the representative can be chosen as the same function.
\end{cor}
\begin{proof}
	The first claim follows directly from \cite[Theorem 1.1]{seb2020}. To see the second, let $\bar f\in \Ne pX$ be a Borel representative of $f$. The set $E:=\{f\ne \bar f \}$ is $p$-exceptional, i.e. $\Gamma_E:=\{ \gamma: \gamma\inv(E)\ne \varnothing \}$ has zero $p$-modulus. Note that, if $f$ is Lipschitz, then $f$ is automatically Borel and we do not need to change representatives, and we can set $\Gamma_E = \emptyset$.
	
	If $\bar A$ is the set in Lemma \ref{lem:nullborel} for $\bar f, g$, then $A:=\bar A\setminus (\Gamma_E\times I)\subset \Diff(f,g)$ and $N:=A^c$ satisfies the claim since it is Borel and $N\subset \Gamma_E\times I\cup \bar A^c$.
	
	The  last claim follows since $N^c$ is Borel and, if $(\gamma,t)\notin N$, we have that $$(f\circ\gamma)_t'=\lim_{n\to\infty}n(f(\gamma_{t+1/n})-f(\gamma_t)).$$
\end{proof}

\subsection{Essential supremum} \label{subsec:esssup}

\begin{defn}\label{def:essup} Let $X$ be a $\sigma$-finite measure space and $\mathcal{F}$ a collection of measurable functions on $X$, then there exists a function $g: X \to \R \cup \{\infty,-\infty\}$ which is measurable, and
	\begin{enumerate}
		\item[A] For each $f \in \mathcal{F}$, $$f \leq g$$ almost everywhere.
		\item[B] For each $g'$ that satisfies [A], will satisfy $g \leq g'$ almost everywhere.
	\end{enumerate}
	
	We call $g =\esssup_{f \in \mathcal{F}} \, f$.
	Similarly, we define $g = \essinf_{f \in \mathcal{F}}\,f$, by switching the directions of the inequalities and assuming $g: X \to \R \cup \{\infty,-\infty\}.$
\end{defn}

We will need the following standard lemma. While its proof is standard, we provide it for the sake of  completeness.

\begin{lemma}\label{lem:essup} If $X$ is any $
	\sigma$-finite measure space and $\mathcal{F}$ is any collection of measurable functions, then $\esssup_{f \in \mathcal{F}} f$ and $\essinf_{f \in \mathcal{F}} f$ exists, and further, there are sequences $f_n,g_n \in \mathcal{F}$ so that $\esssup_{f \in \mathcal{F}} f= \sup_n f_n $ and $\essinf_{f \in \mathcal{F}} f= \inf_n g_n $ almost everywhere.
\end{lemma}
\begin{proof}
	By considering $\{\arctan(f) : f \in \mathcal{F}\}$, we can assume that the collection is bounded. Further, by $\sigma$-finiteness, and after exhausting the space by finite measure sets, it suffices to consider a bounded measure. Define $\mathcal{G}$ to be the collection of all functions of the form $\max(f_1, \dots, f_k)$ for some $f_i \in \mathcal{F}$. By construction, if $g,g' \in G$, then $\max(g,g') \in G'$. 
	
	Consider $U = \sup_{g \in \mathcal{G}} \int g \ud\mu$. There is a sequence $g_n$ so that $\lim_{n \to \infty} \int g_n \ud\mu = U$. By modifying the sequence if necessary, we may take it increasing in $n$, and define $g = \lim_{n \to \infty} g_n$. 
	
	We claim that $g$ is an essential supremum for $\mathcal{F}$. First, if $f 
	\in \mathcal{F}$, and $f> g$ on a positive measure set, then $\lim_{n \to \infty }\int \max(f,g_n) \ud \mu>U$, contradicting the definition of $U$. Thus the condition A in the definition is satisfied. 
	
	Now, if $h$ is any other function satisfying A, then $h \geq g_n$, and thus $h \geq g$ almost everywhere, by construction. Thus B is also satisfied. Finally, the construction gives a countable collection $g_n$ formed each from finitely many $f_i \in \mathcal{F}$, and thus gives the final claim in the statement.
\end{proof}

\bibliographystyle{plain}
\bibliography{abib}

\def\cprime{$'$}
\begin{thebibliography}{10}

\bibitem{almar16}
G.~Alberti and A.~Marchese.
\newblock On the differentiability of {L}ipschitz functions with respect to
  measures in the {E}uclidean space.
\newblock {\em Geom. Funct. Anal.}, 26(1):1--66, 2016.

\bibitem{amb15}
L.~Ambrosio, M.~Colombo, and S.~Di~Marino.
\newblock Sobolev spaces in metric measure spaces: reflexivity and lower
  semicontinuity of slope.
\newblock In {\em Variational methods for evolving objects}, volume~67 of {\em
  Adv. Stud. Pure Math.}, pages 1--58. Math. Soc. Japan, [Tokyo], 2015.

\bibitem{amb13}
L.~Ambrosio, S.~Di~Marino, and G.~Savar{\'e}.
\newblock On the duality between p-modulus and probability measures.
\newblock {\em J. Eur. Math. Soc.}, 17:1817--1853, 2015.

\bibitem{AGS08}
L.~Ambrosio, N.~Gigli, and G.~Savar{\'e}.
\newblock {\em Gradient flows in metric spaces and in the space of probability
  measures}.
\newblock Lectures in Mathematics ETH Z\"urich. Birkh\"auser Verlag, Basel,
  second edition, 2008.

\bibitem{ambgigsav}
L.~Ambrosio, N.~Gigli, and G.~Savar\'{e}.
\newblock Density of {L}ipschitz functions and equivalence of weak gradients in
  metric measure spaces.
\newblock {\em Rev. Mat. Iberoam.}, 29(3):969--996, 2013.

\bibitem{AGS14}
Luigi Ambrosio, Nicola Gigli, and Giuseppe Savaré.
\newblock Calculus and heat flow in metric measure spaces and applications to
  spaces with ricci bounds from below.
\newblock {\em Inventiones mathematicae}, 195:289–391, 06 2014.

\bibitem{APS14}
Luigi Ambrosio, Andrea Pinamonti, and Gareth Speight.
\newblock Tensorization of cheeger energies, the space $h^{1,1}$ and the area
  formula for graphs.
\newblock {\em Advances in Mathematics}, 281, 07 2014.

\bibitem{bate12diff}
D.~Bate.
\newblock Structure of measures in {Lipschitz} differentiability spaces.
\newblock {\em Journal of the American Mathematical Society}, 28(2):421--482,
  2015.

\bibitem{bate2013differentiability}
D.~Bate and G.~Speight.
\newblock Differentiability, porosity and doubling in metric measure spaces.
\newblock {\em Proceedings of the American Mathematical Society},
  141(3):971--985, 2013.

\bibitem{bjo11}
A.~Bj{\"o}rn and J.~Bj{\"o}rn.
\newblock {\em {Nonlinear potential theory on metric spaces}}, volume~17 of
  {\em {EMS Tracts in Mathematics}}.
\newblock European Mathematical Society (EMS), Z{\"u}rich, 2011.

\bibitem{bogachev07}
V.~I. Bogachev.
\newblock {\em Measure theory. Vol. 2}.
\newblock Springer Science \& Business Media, 2007.

\bibitem{che99}
J.~Cheeger.
\newblock {Differentiability of {L}ipschitz functions on metric measure
  spaces}.
\newblock {\em Geom. Funct. Anal.}, 9(3):428--517, 1999.

\bibitem{cheegerkleiner}
J.~Cheeger and B.~Kleiner.
\newblock Differentiability of {L}ipschitz maps from metric measure spaces to
  {B}anach spaces with the {R}adon-{N}ikod\'{y}m property.
\newblock {\em Geom. Funct. Anal.}, 19(4):1017--1028, 2009.

\bibitem{CKS}
J.~Cheeger, B.~Kleiner, and A.~Schioppa.
\newblock Infinitesimal structure of differentiability spaces, and metric
  differentiation.
\newblock {\em Anal. Geom. Metr. Spaces}, 4(1):104--159, 2016.

\bibitem{luvcic2020characterisation}
L.~Danka, E.~Pasqualetto, and T.~Rajala.
\newblock Characterisation of upper gradients on the weighted euclidean space
  and applications.
\newblock {\em preprint (arXiv:2007.11904)}, 2020.

\bibitem{davideb20}
G.~C. David and S.~Eriksson-Bique.
\newblock Infinitesimal splitting for spaces with thick curve families and
  euclidean embeddings.
\newblock {\em preprint (arXiv:2006.10668)}, 2020.

\bibitem{de2016structure}
G.~De~Philippis and F.~Rindler.
\newblock On the structure of $\mathcal{A}$-free measures and applications.
\newblock {\em Annals of Mathematics}, pages 1017--1039, 2016.

\bibitem{garethdimarino}
S.~Di~Marino and G.~Speight.
\newblock The {$p$}-weak gradient depends on {$p$}.
\newblock {\em Proc. Amer. Math. Soc.}, 143(12):5239--5252, 2015.

\bibitem{dimarino19}
S.~{Di Marino} and M.~Squassina.
\newblock New characterizations of sobolev metric spaces.
\newblock {\em Journal of Functional Analysis}, 276(6):1853 -- 1874, 2019.

\bibitem{dunsch}
N.~Dunford and J.~T. Schwartz.
\newblock {\em Linear operators. {P}art {I}}.
\newblock Wiley Classics Library. John Wiley \& Sons, Inc., New York, 1988.
\newblock General theory, With the assistance of William G. Bade and Robert G.
  Bartle, Reprint of the 1958 original, A Wiley-Interscience Publication.

\bibitem{shanmunsemmes}
E.~Durand-Cartagena, S.~Eriksson-Bique, R.~Korte, and N.~Shanmugalingam.
\newblock Equivalence of two bv classes of functions in metric spaces, and
  existence of a semmes family of curves under a 1-poincar{\'e} inequality.
\newblock {\em Advances in Calculus of Variations}, 1(ahead-of-print), 2019.

\bibitem{seb2020}
S.~Eriksson-Bique.
\newblock Density of lipschitz functions in energy.
\newblock {\em preprint (arXiv:2012.01892)}, 2020.

\bibitem{exnerova2019plans}
V.~H. Exnerov{\'a}, O.~F.K. Kalenda, J.~Mal{\`y}, and O.~Martio.
\newblock Plans on measures and am-modulus.
\newblock {\em preprint (arXiv:1904.04527)}, 2019.

\bibitem{fuglede1957}
B.~Fuglede.
\newblock Extremal length and functional completion.
\newblock {\em Acta Math.}, 98:171--219, 1957.

\bibitem{gig15}
N.~Gigli.
\newblock On the differential structure of metric measure spaces and
  applications.
\newblock {\em Mem. Amer. Math. Soc.}, 236(1113):vi+91, 2015.

\bibitem{gig18}
N.~Gigli.
\newblock Nonsmooth differential geometry---an approach tailored for spaces
  with {R}icci curvature bounded from below.
\newblock {\em Mem. Amer. Math. Soc.}, 251(1196):v+161, 2018.

\bibitem{gigenr}
N.~Gigli and E.~Pasqualetto.
\newblock {\em Lectures on Nonsmooth Differential Geometry}.
\newblock SISSA Springer Series (Volume 2). 2020.

\bibitem{haj96}
P.~Haj{\l}asz.
\newblock {Sobolev spaces on an arbitrary metric space}.
\newblock {\em Potential Anal.}, 5(4):403--415, 1996.

\bibitem{haj03}
P.~Haj{\l}asz.
\newblock {Sobolev spaces on metric-measure spaces}.
\newblock In {\em {Heat kernels and analysis on manifolds, graphs, and metric
  spaces ({P}aris, 2002)}}, volume 338 of {\em {Contemp. Math.}}, pages
  173--218. Amer. Math. Soc., Providence, RI, 2003.

\bibitem{hei98}
J.~Heinonen and P.~Koskela.
\newblock {Quasiconformal maps in metric spaces with controlled geometry}.
\newblock {\em Acta Math.}, 181(1):1--61, 1998.

\bibitem{HKST07}
J.~Heinonen, P.~Koskela, N.~Shanmugalingam, and J.~Tyson.
\newblock {\em {Sobolev spaces on metric measure spaces: an approach based on
  upper gradients}}.
\newblock {New Mathematical Monographs}. Cambridge University Press, United
  Kingdom, first edition, 2015.

\bibitem{tet2020}
T.~Ikonen, E.~Pasqualetto, and E.~Soultanis.
\newblock Abstract and concrete tangent modules on lipschitz differentiability
  spaces.
\newblock {\em preprint (arXiv:2011.15092)}, 2020.

\bibitem{keith03}
S.~Keith.
\newblock Modulus and the {P}oincar\'e inequality on metric measure spaces.
\newblock {\em Mathematische Zeitschrift}, 245(2):255--292, 2003.

\bibitem{keith04}
S.~Keith.
\newblock A differentiable structure for metric measure spaces.
\newblock {\em Advances in Mathematics}, 183(2):271 -- 315, 2004.

\bibitem{kei04}
S.~Keith.
\newblock {Measurable differentiable structures and the {P}oincar{\'e}
  inequality}.
\newblock {\em Indiana Univ. Math. J.}, 53(4):1127--1150, 2004.

\bibitem{mac13}
J.~M. Mackay, J.~T. Tyson, and K.~Wildrick.
\newblock {Modulus and {P}oincar{\'e} inequalities on non-self-similar
  {S}ierpi{\'n}ski carpets}.
\newblock {\em Geom. Funct. Anal.}, 23(3):985--1034, 2013.

\bibitem{pasqualetto20b}
E.~Paqualetto.
\newblock Testing the sobolev property with a single test plan.
\newblock {\em preprint (arxiv:2006.03628)}, 2020.

\bibitem{rud}
W.~Rudin.
\newblock {\em Function theory in the unit ball of {$\Bbb C^n$}}.
\newblock Classics in Mathematics. Springer-Verlag, Berlin, 2008.
\newblock Reprint of the 1980 edition.

\bibitem{sch16b}
A.~Schioppa.
\newblock Derivations and alberti representations.
\newblock {\em Advances in Mathematics}, 293:436--528, 2016.

\bibitem{sch16}
A.~Schioppa.
\newblock {Metric currents and {A}lberti representations}.
\newblock {\em J. Funct. Anal.}, 271(11):3007--3081, 2016.

\bibitem{sem96}
S.~Semmes.
\newblock {Finding curves on general spaces through quantitative topology, with
  applications to {S}obolev and {P}oincar{\'e} inequalities}.
\newblock {\em Selecta Math. (N.S.)}, 2(2):155--295, 1996.

\bibitem{sha00}
N.~Shanmugalingam.
\newblock {Newtonian spaces: an extension of {S}obolev spaces to metric measure
  spaces}.
\newblock {\em Rev. Mat. Iberoamericana}, 16(2):243--279, 2000.

\end{thebibliography}
\end{document}